\documentclass[12pt]{article}

\usepackage[english]{babel}

\usepackage[top=2.5cm,bottom=2.5cm,left=3cm,right=3cm,marginparwidth=1.75cm]{geometry}

\usepackage{amsmath}
\usepackage{amsthm}
\usepackage{mathtools}
\usepackage{amsfonts}
\usepackage{amssymb}
\usepackage{wasysym}
\usepackage{bbm}
\usepackage{mathbbol} 
\usepackage{xargs}
\usepackage{commath} 

\usepackage{graphicx}
\usepackage[colorlinks=true, allcolors=blue]{hyperref}
\usepackage[noabbrev,nameinlink,capitalise]{cleveref}
\usepackage{thmtools} 
\usepackage{cite}

\usepackage{float}
\usepackage{calc}
\usepackage{tikz}
\usepackage{tikz-cd}
\usetikzlibrary{graphs,arrows,decorations.pathmorphing,decorations.markings,fit,positioning,hobby,arrows.meta}
\usepackage{subcaption}

\theoremstyle{definition}
\newtheorem{theorem}{Theorem}[section]

\newtheorem{rem}[theorem]{Remark}
\newtheorem{lemma}[theorem]{Lemma}
\newtheorem{proposition}[theorem]{Proposition}
\theoremstyle{definition}
\newtheorem{definition}[theorem]{Definition}
\newtheorem{notation}[theorem]{Notation}
\newtheorem{example}[theorem]{Example}

\newcommand{\Dgm}{\mathrm{PD}} 
\newcommand{\birth}{b} 
\newcommand{\death}{d} 
\newcommand{\Mch}{\mathcal{M}} 
\newcommand{\define}[1]{{\bf \boldmath{#1}}}
\newcommand{\R}{\mathbb{R}}
\newcommand{\bS}{\mathbb{S}}

\newcommandx{\pmv}[2][1= X, 2= \mathbf{v}]{\mathcal{#1}(#2)} 
\newcommand{\ch}[1][M]{C(#1)} 
\newcommand{\V}{{\bf v}}

\newcommand{\ip}[2]{\langle #1, #2 \rangle} 
\newcommand{\image}{\mathrm{img}}
\newcommand{\PHT}{\text{PHT}}

\title{Decomposing the Persistent Homology Transform of Star-Shaped Objects}

\author{Shreya Arya\footnote{Duke University}, Barbara Giunti\footnote{SUNY Albany}, Abigail Hickok\footnote{Yale University}, Lida Kanari\footnote{Organization École Polytechnique Fédérale de Lausanne (EPFL), Address Geneva, Switzerland, Division: Blue Brain Project, and Department of Mathematics, University of Oxford} \\ Sarah McGuire Scullen\footnote{Michigan State University}, Katharine Turner\footnote{
Australian National University, Mathematical Sciences Institute}}

\date{}

\begin{document}
\maketitle

\begin{abstract}
In this paper, we study the geometric decomposition of the degree-$0$ Persistent Homology Transform (PHT) as viewed as a persistence diagram bundle. 
We focus on star-shaped objects as they can be segmented into smaller, simpler regions known as ``sectors". 
Algebraically, we demonstrate that the degree-$0$ persistence diagram of a star-shaped object in $\R^2$ can be derived from the degree-$0$ persistence diagrams of its sectors. 
Using this, we then establish sufficient conditions for star-shaped objects in $\R^2$ so that they have ``trivial geometric monodromy''. 
Consequently, the PHT of such a shape can be decomposed as a union of curves parameterized by $\bS^1$, where the curves are given by the continuous movement of each point in the persistence diagrams that are parameterized by $\bS^{1}$. 
Finally, we discuss the current challenges of generalizing these results to higher dimensions.
\end{abstract}

\section{Introduction}
Topological data analysis (TDA) ~\cite{Carlsson2009,wasserman} is a discipline studying the increasingly complex datasets that are becoming available in different scientific disciplines. 
Specifically, topological persistence allows the stable representation of data and can be used to separate noise from signal. 

The Persistent Homology Transform (PHT for short), introduced in ~\cite{Turner2014}, is a method from TDA that describes a shape in $\R^d$ by a set of \textit{persistence diagrams}. 
These diagrams are obtained from the sublevelset filtrations constructed by ``scanning'' the shape in different directions.
It was shown in \cite{curry2022many} and \cite{ghrist2018persistent} that this assignment of shapes to a set of persistence diagrams is injective. 
In other words, the PHT is a statistical tool that sufficiently describes shapes in $\R^d$.

One of the advantages of PHT is that it can be used to define an informative metric on the shape space so that one can identify morphological differences among a collection of shapes. 
The PHT is especially useful for the computational modeling of $2D$ and $3D$ shapes; 
previous applications include the study of primate heel bones~\cite{Turner2014}, shape classification of binary 2D images \cite{Hofer2017shapeclassification,landi_frosini02,ferri_frosini95}, and prediction of clinical outcomes in patients~\cite{Crawford2016}. 

The PHT of a shape can be viewed as a persistence diagram bundle \cite{Hickok2022}. 
Persistence diagram bundles are analogous to fibre bundles but the fibres at each location are replaced with persistence diagrams. 
Given a flat connection on a fibre bundle we have a notion of how the fibres locally connect to each other. 
For any loop in the base space we can study the corresponding action on the fibre above its starting location as we use the connection to lift the loop from the base space into the fibres.  
Monodromy is the name for the action. 
This monodromy is trivial if we always return to the same element. 
Equivalently, the monodromy is trivial when there exists a set of global sections which form a basis when restricted to any single fibre. 
Persistence diagram bundles are not fibre bundles but do have similarities. 
Continuity of the persistence diagrams implies continuity of the paths of the points within the persistence diagrams. 
This is only well-defined as stated if we have multiplicity at most one for any off-diagonal points. 
We call these objects \textit{simple persistence diagrams}. 
For the purposes of this paper, we will restrict to persistence diagram bundles which only contain simple persistence diagrams. 
We can then adapt the notion of trivial monodromy to persistence diagram bundles.

The objective of this paper is to study the monodromy of the degree-$0$ PHT of a shape. 
In particular, we restrict our attention to planar star-shaped objects (see \cref{sec:star_shaped_decomposition}). 
We first prove (\Cref{thm:sector_decomposition}) that the PHT of a planar star-shaped object $M$ can be decomposed into the direct sum of the PHTs of each ``sector'' (see definition in~\cref{sec:star_shaped_decomposition}) of $M$. 
The benefit of our decomposition is that the PHT of a single sector is much simpler than the PHT of the entire shape (see \Cref{lem:finite_support_Mi}, for example). 
Using the decomposition from \Cref{thm:sector_decomposition}, we then prove (\Cref{thm:no_monodromy}) that the PHT of a planar star-shaped polygon in general position has trivial monodromy. 
This result holds more generally for star-shaped planar polygons whose persistence diagrams have no points of multiplicity greater than $1$ (see \Cref{prop:no_monodromy_suff_conditions}). 
This is particularly useful in applications: inputs are often encoded into meshes to be processed by a machine; in other words, they are polygons. 
Moreover, the different components of each input are often star-shaped. 
To the best of our knowledge, our paper is the first to discover a relationship between the geometry of the PHT of $M$ and the geometry of the shape $M$ itself.
 
Star-shaped objects are a large class of shapes that appear often in real-world data. 
For example, plant morphology data---specifically, leaf shapes---can often be represented as star-shaped objects. 
In \cite{Chitwood2017leaves}, the authors presented a dataset of approximately $3,300$ leaf scans, representing $40$ different Maracuy\'a (also commonly referred to as Passiflora) species.
Most of the leaf samples can be represented as star-shaped objects, with the exception of only a couple of species (e.g., samples from the \textit{P. morifolia} and \textit{P. cincinnata} species). 
\cref{fig:leaf_examples} shows sample representations of different species from this dataset, which satisfy the star-shaped conditions that we consider for the remainder of this paper. 
Our polygonal assumption is also reasonable for real-life objects as they can be approximated by polygons or polyhedra of increasingly many edges/faces so that they are arbitrarily close to the true shape. 

\begin{figure}[h!]
    \centering
    \includegraphics[width = .8\textwidth]{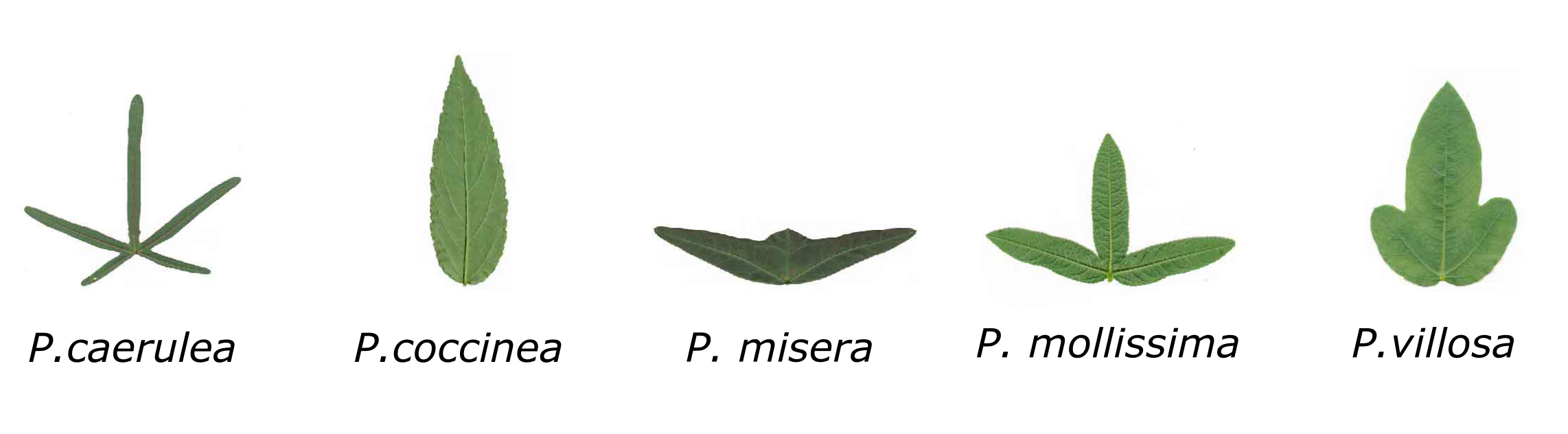}
    \caption{Sample leaf scans from the Maracuy\'a leaf shape dataset of five different species, each satisfying the conditions of a star-shaped object \cite{Chitwood2017leaves}.}
    \label{fig:leaf_examples}
\end{figure}

\section{Background and Preliminaries}

Homology is an algebraic way to describe topological 
features \cite{hatcher}. 
Although homology groups can be defined and computed in all positive integer degrees, this paper will focus only on degree-$0$ homology. 
As such, we will restrict the background definitions to degree zero.

\begin{definition}
  Let $M \subset \R^d$. 
  The \define{degree-$0$ homology of $M$} over a field $\mathbb{F}$, denoted by $H_0(M, \mathbb{F})$, can be defined as the vector space over $\mathbb{F}$ of formal linear combinations of connected components of $M$. 
  The elements of $H_0(M, \mathbb{F})$ are degree-$0$ homology classes.
\end{definition}

For the remainder of the paper, we fix a field $\mathbb{F}$ and omit it from our notation. We let $[x]_M$ denote the connected component that contains $x$.
\medskip

Reduced homology is a modification of ordinary homology theory that was designed so that the reduced homology groups of a point are all zero. 

\begin{definition}
  Let $M \subset \R^d$.
  The \define{degree-$0$ reduced homology of $M$}, written as $\widetilde{H}_0(M)$, can be defined as the vector space of formal linear combinations of connected components whose coefficients sum to $0$.    
\end{definition}
\noindent That is, elements of $\widetilde{H}_0(M)$ are of the form $\sum_i \lambda_i [x_i]_M$ with $\sum_i \lambda_i=0$. For all non-empty $M$, the reduced homology satisfies 
\begin{equation*}
    \dim(\widetilde{H}_0(M)) = \dim(H_0(M)) - 1\,.
\end{equation*}

For the statements of the main results, the reader needs only to be familiar with degree-$0$ homology, but in some proofs, we will use the degree-$1$ homology of $M$, denoted $H_1(M)$. 
For interested readers unfamiliar with $H_1$, we recommend \cite{hatcher}. 
For the purposes of this paper, the only fact we need is that $H_1(M)$ is zero if $M$ is contractible (i.e. if $M$ can be continuously deformed into a point).

An increasing family of subsets $\{M_t\}_{t\in \R}$ (i.e., a collection of subsets of $\R^d$ such that $M_s\subset M_t$ for all $s\leq t$) is called a \define{filtration}. 
Given a filtration, we can study how the homology evolves as $t$ increases. For each $M_s \subset M_t$, the inclusion map $M_s \xhookrightarrow{} M_t$ induces a linear map between the homology groups  
\begin{align*}
    \iota_s^t\colon H_0(M_s)&\to H_0(M_t)\\
    \sum_i \lambda_i[x_i]_{M_s} &\mapsto \sum_i \lambda_i[x_i]_{M_t} \, .
\end{align*}
This is well-defined because if $[x]_{M_s} = [y]_{M_s}$, then $[x]_{M_t} = [y]_{M_t}$. 
This same construction works for reduced homology groups as well because the sum of the coefficients remains zero.

\begin{definition}
    A \define{persistence module} is a pair $(\{V_t\}_{t \in \R}, \{\phi_{st}\}_{s \leq t})$, where $\{V_t\}$ is a set of vector spaces over a field $\mathbb{F}$ and $\{\phi_{st}\}$ is a set of linear maps $\phi_{st} \colon V_s \to V_t$ such that $\phi_{st} \circ \phi_{rs} = \phi_{rt}$ for all $r \leq s \leq t\in\R$ and $\phi_{tt}$ is the identity over $M_t$.
\end{definition}

\noindent In other words, a persistence module is a functor $X \colon \R \to \mathbf{Vect}$, where $\mathbf{Vect}$ is the category of vector spaces over a field $\mathbb{F}$. 
Typically, persistence modules are obtained from the homology of a filtration.

\begin{definition}
Let $\{M_t\}_{t\in \R}$ be a filtration. 
The \define{persistent homology} in degree 0 over a field $\mathbb{F}$ is the persistence module $(\{H_0(M_t)\}_{t \in \R}, \{\iota_s^t\}_{s \leq t})$, where $\iota_s^t \colon H_0(M_s) \to H_0(M_t)$ is the map induced by the inclusion $M_s \hookrightarrow M_t$.
\end{definition}

\begin{definition}(\hspace{5sp}\cite[Chapter VII]{edel_harer2010computational}) Let $\{M_t\}_{t \in \R}$ be a filtration. 
The \define{reduced persistent homology} in degree 0 over a field $\mathbb{F}$ is the persistence module $(\{\widetilde{H}_0(M_t)\}_{t\in\R}, \{\iota_s^t\}_{s \leq t})$, where $\widetilde{H}_0(M_t)$ denotes the reduced homology of $M_t$ and $\iota_s^t \colon \widetilde{H}_0(M_s) \to \widetilde{H}_0(M_t)$ is the map induced by the inclusion $M_s \hookrightarrow M_t$.
\end{definition}

\begin{definition}
Let $X = (\{X_t\}, \{\phi^X_{st}\})$ and $Y = (\{Y_t\}, \{\phi^Y_{st}\})$ be two persistence modules.
An \define{isomorphism of persistence modules} is a set $\{f_t\colon X_t\to Y_t\}_{t\in\R}$ of linear maps such that $f_t$ is an isomorphism of vector spaces for all $t \in \R$ and $f_{t}\circ\phi^{X}_{s t}=\phi^{Y}_{s t}\circ f_s$ for all $s\leq t\in\R$.
If there exists an isomorphism between $X$ and $Y$, they are said to be \define{isomorphic}.
\end{definition}

The building blocks of persistence modules are interval modules.
Recall that an \define{interval} in $\R$ is a non-empty, connected subset of $\R$.
    
\begin{definition}
    Let $A \subset \R$ be an interval. 
    The \define{interval module} with support $A$ (over the field $\mathbb{F}$) is the persistence module $(\{V_t\}, \{\phi_{st}\})$ such that $V_t =\mathbb{F}$ for $t \in A$ and $V_t =0$ for $t\notin A$ and such that the transition maps $\phi_{st}$ are the identity for $s \leq t \in A$. 
    We denote this interval module by $\mathcal{I}_A$.
\end{definition}
    
Our assumptions on the shapes of interest will result in their persistence module(s) $X = (\{V_t\}, \{\phi_{st}\})$ being pointwise finite-dimensional (p.f.d.); i.e., $\dim_{\mathbb{F}}(V_t)<\infty$ for all $t\in\R$. Therefore, we may apply the following theorem.
    
\begin{theorem}(\hspace{5sp}\cite[Theorem 1.1]{Crawley_Boevey_2015})\label{thm:intervaldecomposition}
    A p.f.d. persistence module $X$ admits an \define{interval decomposition}. 
    That is, the module can be decomposed as follows:
    \[
    X \cong \bigoplus\limits_{A \in \mathcal{S}} \mathcal{I}_A \,,
    \]
    where $\mathcal{S}$ is a set of intervals $A$ in $\R$. Moreover, this decomposition is unique up to reorderings of the intervals.
\end{theorem}

The interval decomposition of degree-$0$ persistent homology tells us the filtration values $t$ at which new connected components appear or old connected components merge. 
For example, suppose that $X = (\{H_0(M_t)\}, \{\iota_{st}\})$ is the degree-$0$ persistent homology of a filtration $\{M_t\}$. 
Let $\mathcal{I}_A$ be an interval module in the interval decomposition of $X$, with $b = \inf(A)$ and $d = \sup(A)$. 
This indicates that a new connected component first appears (``is born'') in $M_b$ and then ``dies'' (i.e., merges with another connected component) in $M_d$.

\begin{notation}
    Let $\overline{\R}\coloneqq \R\cup \{+\infty\}$, $\overline{\R}^{2}\coloneqq \overline{\R}\times\overline{\R}$, and $\Delta\coloneqq\{(x,x)\in\overline{\R}^2\}$. 
\end{notation}
    
\begin{definition}
    Let $X$ be a persistence module with interval decomposition $X \cong \bigoplus_{A \in \mathcal{S}} \mathcal{I}_A$. We can associate to each interval module $\mathcal{I}_A$ a point in $\overline{\R}^2$ with the first coordinate $\birth(\mathcal{I}_A)\coloneqq\inf(A)$ and second coordinate $\death(\mathcal{I}_A)\coloneqq\sup(A)$. 
    We refer to $\birth(\mathcal{I}_A)$ as the \define{birth} time and $\death(\mathcal{I}_A)$ as the \define{death} time.
    By taking the union over all the intervals in the interval decomposition, we obtain a multiset of points in $\overline{\R}^{2}$. We then add to this multiset a single copy of $\Delta$ which represents an abstract copy of the diagonal.
    We refer to this final multiset as the \define{persistence diagram}, and we denote it by $\Dgm(X)$ for a persistence module $X$.
\end{definition}

Points of the form $(\birth, \infty)$
in $\Dgm(X)$ for $\birth \in \R$ are called \define{essential classes} and points of the form $(\birth, \death)$ where $\death\neq \infty$ are called \define{non-essential classes}. 
In other words, non-essential classes are those that die within the filtration, and essential classes are those that are born within the filtration and never die. 

\begin{rem}
Let $\{M_t\}_{t\in\R}$ be a filtration. Let $X$ and $\widetilde{X}$ be the unreduced and reduced persistent homology of $\{M_t\}$, respectively. 
Then their interval decompositions differ only by one interval module, namely $\mathcal{I}_{[b,\infty)}$, where $b$ is a minimum among the births of the essential classes of $X$. In other words, $X = \widetilde{X} \oplus \mathcal{I}_{[b, \infty)}$.
\end{rem}

\begin{definition} 
Let $\Dgm(X)$ and $\Dgm(Y)$ be persistence diagrams represented by countable multisets in $\overline{\R}^2$. A \define{matching} $\Mch$ between $\Dgm(X)$ and $\Dgm(Y)$ is a subset of $\Dgm(X) \times \Dgm(Y)$ such that every element in $\Dgm(X)\backslash \Delta$ and $\Dgm(Y) \backslash \Delta$ appears in exactly one pair.
\end{definition}
\noindent In the above definition, the abstract diagonal element $\Delta$ may appear in many pairs. 
The matching is used to define a distance as follows.
As the non-diagonal elements in $\Dgm(X)$ and $\Dgm(Y)$ lie in $\overline{\R}^2$, we can use the $l_\infty$ distance in the plane to define a metric on $\overline{\R}^2$. We extend $\ell_{\infty}$ to $\overline{\R}^2\cup \Delta$ by defining:
\begin{align*}
    \| (a,b)-\Delta\|_\infty &=\frac{b-a}{2}\,, \\
    \| \Delta - \Delta \|_\infty &= 0 \,, \\
    \| (a, \infty) - (b, \infty)\|_\infty &= |a - b|\,, \\
    \|(a,\infty)-x\|_\infty &=\infty
\end{align*}
for all $a, b \in \R$ and $x \in \R^2 \cup \Delta$. 
In other words, if $x$ is an off-diagonal point, the pair $(x,\Delta)$ acts as the pair $(x,x')$, where $x'$ is the closest point to $x$ on $\Delta$.
\medskip

The bottleneck distance, defined below, is an extended metric on the space $\mathcal{PD}$ of persistence diagrams.

\begin{definition}
Let $X$ and $Y$ be two persistence modules. 
The \define{bottleneck distance} $d_B$ between them is defined by
\begin{equation}
  d_B(\Dgm(X), \Dgm(Y)) \coloneqq \inf_{\Mch} \sup_{(x,y)\in \Mch} ||x-y||_\infty
\end{equation}
  where $\Mch \subset  \Dgm(X)  \times \Dgm(Y)$ is a matching.
\end{definition}

We conclude by recalling a standard stability result in persistence theory:

\begin{theorem}[\hspace{5sp}\cite{OPstability}] 
Let $T$ be a triangulable space and $f,g \colon T \rightarrow \R$. 
Then the persistence diagrams $\Dgm(f)$ and $\Dgm(g)$ for their sublevel set filtrations satisfy
\[
d_B(\Dgm(f),\Dgm(g)) \leq ||f -g||_\infty \, .
\]
\end{theorem}

\subsection{The Persistent Homology Transform}

\begin{definition}\label{ex: height function}
Let $M\subset \R^d$ and $\V\in\mathbb{S}^{d-1}$. 
The \define{height function} of $M$ with respect to $\V$ is $h_{\V}\colon M\to \R$, where $h_{\V}(x)\coloneqq \ip{\V}{x}$. 
In other words, $h_\V(x)$ is the scalar product in $\R^d$ of $x$ and $\V$.
The \define{sublevelset filtration} of $M$ with respect to $\V$ is the filtration given by $\{h_\V^{-1}((-\infty, t])\}_{t\in\R}$.
For the unreduced and reduced homology, the associated degree-$0$ persistence modules are, respectively, $(\{H_0(h_\V^{-1}(-\infty, t])\}_{t\in\R}, \{\iota_s^t\}_{s \leq t})$ and $(\{\widetilde{H}_0(h_\V^{-1}(-\infty, t])\}_{t\in\R}, \{\tilde{\iota}_s^t\}_{s \leq t})$, where $\iota_s^t$ and $\tilde{\iota}_s^t$ are the maps on unreduced and reduced homology, respectively, that are induced by the inclusion $M_s \xhookrightarrow{} M_t$, where $M_s=\{x\in M \mid \langle x, \V\rangle \le s\}$. 
\end{definition}

The \define{Persistent Homology Transform} (PHT as defined in~\cite{Turner2014}) describes a shape in $\R^d$ with a set of persistence diagrams indexed by directions $\V \in \bS^{d-1}$. 
In more detail,
the PHT of a shape $M\subset \R^d$ is a map that associates each direction $\V$ in the unit sphere with the persistence diagram obtained by the sublevelset filtration (\cref{ex: height function}) in that direction. 
 
The degree-$0$ persistence diagram associated with the height function in direction $\V$ is denoted by $\Dgm_0(M, \V)$. 
It summarises the changes in the topology of $M$ filtered in direction $\V$. 
More precisely, let $h_\V$ be the height function in \cref{ex: height function}, and let $M_{\V, t}$ be the $t$-sublevel set of $h_\V$. In other words,
\begin{equation}\label{eq:sublevel}
    M_{\V, t} \coloneqq \{x \in M \mid \ip{x}{\V} \leq t\}\, .
\end{equation}
For each $\V \in \bS^{d-1}$, the set $\{M_{\V, t}\}_{t \in \R}$ is a filtration of $M$. We denote the degree-$0$ persistent homology of $\{M_{\V, t}\}$ by $\text{PH}_0(M, \V)$. 
The diagram $\Dgm_0(M, \V)$ is the persistence diagram for $\text{PH}_0(M, \V)$. 
Similarly, we denote the reduced persistent homology of $\{M_{\V, t}\}$ by $\widetilde{\text{PH}}_0(M, \V)$, and we denote its persistence diagram by $\widetilde{\Dgm}_0(M, \V)$.

\begin{definition}
    The degree-$0$ \define{persistent homology transform} of $M$, denoted by $\PHT_0(M)$, is defined as the continuous map from the $(d-1)$-sphere to the space of persistence diagrams:
    \begin{align*}
        \PHT_0(M)\colon \bS^{d-1}&\to \mathcal{PD} \\
        \V &\mapsto \Dgm_0(M, \V)  \, .
    \end{align*}
\end{definition}

\begin{rem}\label{rem: LipschitzPHT}
    It was shown in \cite{Turner2014} that the map $\PHT_0(M)$ is Lipschitz with respect to the Euclidean metric on $\bS^{d-1}$ and the bottleneck distance on the space of persistence diagrams, with Lipschitz constant equal to $K\coloneqq\sup_{x\in M}\|x\|$. 
\end{rem}

\subsection{Geometric monodromy of persistence diagram bundles}\label{sec:monodromy_def}

\begin{definition}\label{def:pdb}
A \define{persistence diagram bundle} consists of a topological space $B$ (the \define{base space}) and a function $f\colon B\to \mathcal{PD}$ which is continuous with respect to the bottleneck distance. 
The \define{total space} of the persistence diagram bundle is $E\coloneqq \{(p, z) \mid p \in B, z \in f(p)\}$. 
It has the subspace topology induced by $E \subset B \times \mathcal{PD}$. 
\end{definition}

\begin{example}
    Given a shape $M \subset \R^d$, the pair $(\bS^{d-1}, \PHT_0(M))$ is a persistence diagram bundle with base space $\bS^{d-1}$.
\end{example}

Paths in $B$ lift to paths in the total space $E$ if the persistence diagram bundle is ``simple,'' as defined below. 
\begin{definition}\label{def:simple_pdb}
    We say that the persistence diagram bundle $(B, f)$ is \define{simple} if, for all $p \in B$, the persistence diagram $f(p)$ has no points with multiplicity greater than or equal to $2$.
\end{definition}

Star-shaped polygons in $\R^2$ with vertices in general position are one family of shapes whose (PHT) persistence diagram bundle is simple. 
More generally, many shapes can satisfy this property, including polygons, star-shaped objects, and trees. While this property is not guaranteed by default for any shape family, an $\epsilon$-perturbation of the original shape can often ensure the uniqueness of the points in the persistence diagram bundle, thereby making it simple. A more detailed example of a class of objects satisfying this condition can be found in \cref{sec:monodromy}, and additional examples are illustrated in \cref{fig:spiral_and_5armstar} and discussed in \cref{ex_monodromy}.
\medskip

\noindent Let $(B, f)$ be a simple persistence diagram bundle. 
Suppose we fix some $p\in B$ and a continuous path $\gamma\colon[0,1]\to B$ that starts at $p$. 
Then, for every off-diagonal point $q$ in  $f(p)$, we can use continuity to lift the path $\gamma$ to a path $\widetilde{\gamma} \colon [0, 1] \to E$ such that $\widetilde{\gamma}(0) = (q, p)$ and $\pi \circ \widetilde{\gamma} = \gamma$, where $\pi$ is the projection $(\overline{\R}^2 \cup \Delta) \times B \to B$.
If we assume that the diagonal is absorbing (i.e., there is an interval $[t_0, T] \subset [0, 1]$ such that $\widetilde{\gamma}(t) \neq \Delta$ if and only if $t \in [t_0, T]$), then the path $\widetilde{\gamma}$ is unique.
\medskip

We wish to adapt the notion of monodromy within fibre bundles to persistence diagram bundles. 
In this paper, we only consider the question of whether or not the monodromy is trivial. 
Therefore, we choose to avoid giving a technical definition of the monodromy action and instead just provide a definition for trivial monodromy. 
In this, we follow the approach of studying monodromy in the context of TDA as described in ~\cite{Cerri2019,Cerri2013,Hickok2022}. 

Note that since we are only using the continuity of the persistence diagrams, our notion of lift is purely geometric, and as such we will include the adjective ``geometric'' wherever appropriate. 
If we were using persistence modules and interleaving maps then we would instead have algebraic relationships and the monodromy would then be of an algebraic flavour. 
This is beyond the scope of this paper. 

Before defining ``trivial geometric monodromy,'' we must first define sections.

\begin{definition}
    A \define{section} within a persistence diagram bundle $(B,f)$ is a continuous function $\gamma \colon B \to \overline{\R}^2\cup \Delta$ such that $\gamma(p) \in f(p)$ for all $p\in B$ and such that $\{p \in B \mid \gamma(p) \neq \Delta\}$ is connected.
\end{definition}

In the above definition, we made the diagonal explicit, even if $\Delta\subset\overline{\R}^2$, to underline the different behavior of points on and off the diagonal.

\begin{definition}\label{def:no_geometric_monodromy}
A persistence diagram bundle $(B, f)$ has \define{trivial geometric monodromy} if there exist sections $\{\gamma_i \colon B  \to \overline{\R}^2 \cup \Delta\}$ such that $f(p) = \bigcup_{i \mid \gamma_i(p) \neq \Delta} \gamma_i(p)$ for all $p \in B$.
\end{definition}

\begin{example}
The persistent homology transform of any convex set has trivial geometric monodromy. 
This is because there is only one section, corresponding to the essential class. 
\end{example}

\begin{example}\label{ex_monodromy}
\cref{fig:spiral} shows an example of a shape that has non-trivial geometric monodromy. 
There is one section encoding the essential connected component of $M$. 
However, there cannot be a single section encoding the other connected component(s) of $M$ over the whole circle because it would not satisfy the conditions to be a function. 
Furthermore, a union of sections cannot encode the other connected component(s) because it would violate continuity.
\end{example}

\begin{figure}[h]
\centering
\begin{subfigure}[t]{0.58\textwidth}
    \centering
    \includegraphics[width = \textwidth]{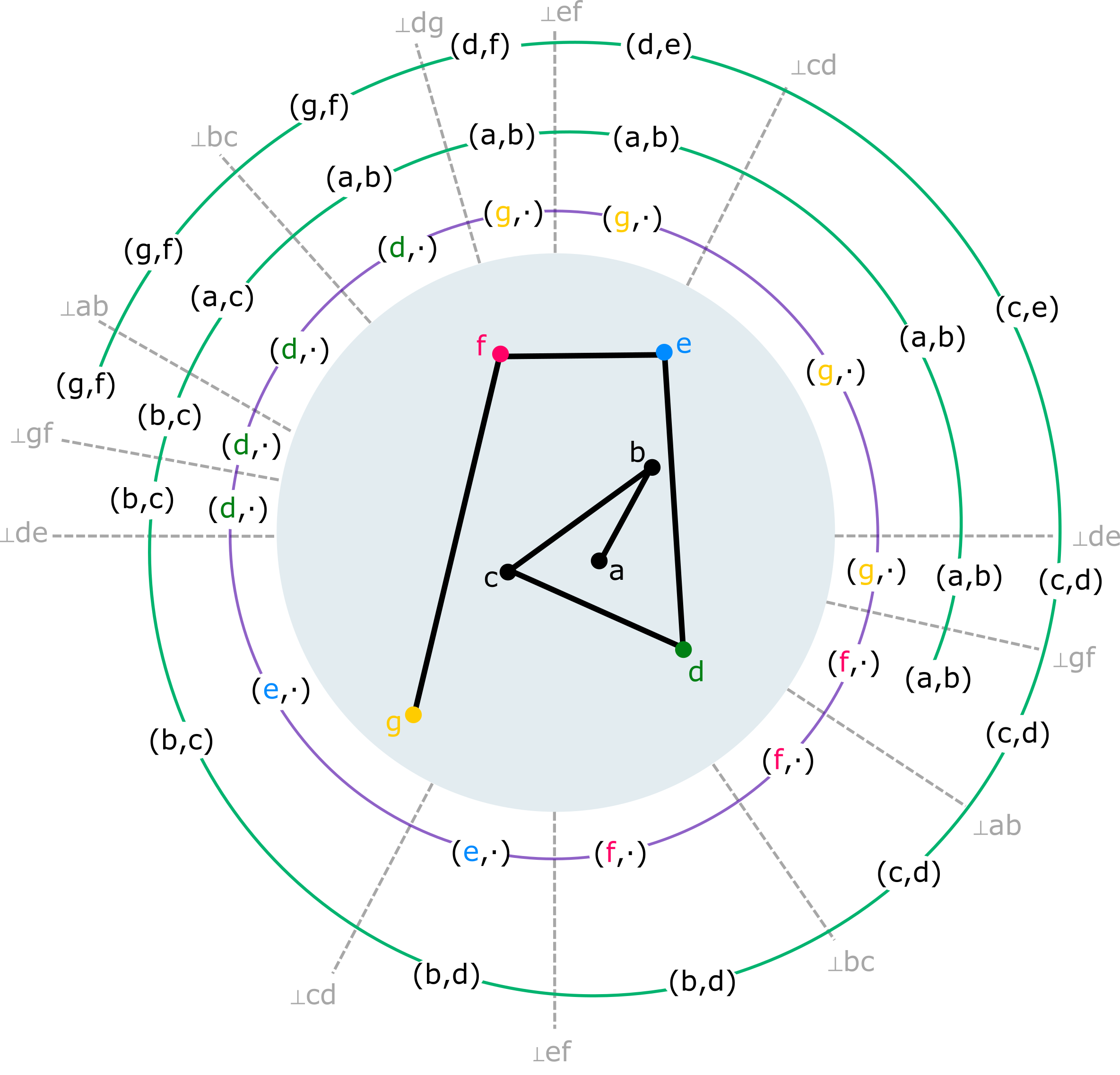}
    \caption{PHT for a spiral.}
    \label{fig:spiral}
\end{subfigure}%
\begin{subfigure}[t]{0.41\textwidth}
    \centering
    \includegraphics[width = \textwidth]{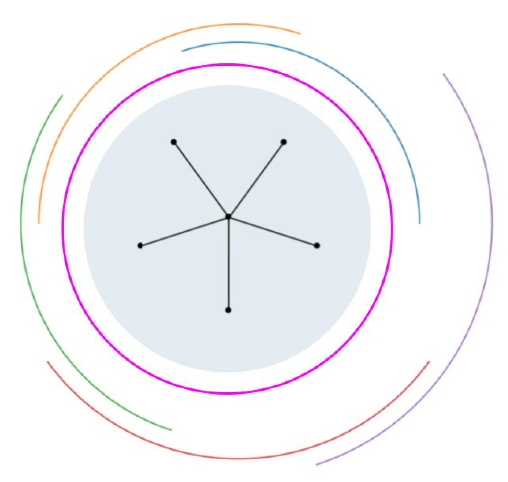}
    \caption{PHT for a star-shaped object.}
    \label{fig:5armstar}
\end{subfigure} 
\caption{Two objects and their associated PHTs, with arcs of $\bS^1$ representing persistent homology classes for directions $\V\in\bS^{1}$. In the case of the spiral-like object (a), the connected components are labeled according to birth-death pairs. The spiral exhibits non-trivial geometric monodromy in its PHT (there is one homological class that ``wraps around,'' so that its representatives cannot be chosen consistently), while the star (b) shows trivial geometric monodromy in its PHT.}
\label{fig:spiral_and_5armstar}
\end{figure}

\subsection{Star-shaped objects and sectors}\label{sec:star_shaped_decomposition}

In this section, we formally define star-shaped objects (\cref{fig:sector_notation,fig:holed_starshaped}), which are the focus of this paper.
We use the convex hull to define the ``sectors'' of a star-shaped object (\cref{fig:sector_notation}). 

\begin{definition}
    A set $M\subset \R^d$ is \define{star-shaped} if there is a point $c \in M$ such that for every $x \in M$, the line segment between $c$ and $x$ is contained in $M$. 
    We call $c$ a \define{center} of $M$. 
\end{definition}

\begin{example}\label{ex_convex}
    Convex sets in $\R^d$ are star-shaped. 
    Moreover, every point in a convex set is a center.
\end{example}

Let $M\subset\R^d$ be star-shaped.
We denote by $\ch$ its convex hull. 
\medskip

\noindent\fbox{%
\parbox{\textwidth}{%
    Throughout the paper, if $M\subset \R^d$ is a star-shaped object, we assume that $\ch$ is a $d$-dimensional convex polytope -- i.e., $\ch$ is the convex hull of a finite set of points in $\R^d$ \cite{Grünbaum2003}. 
}%
}
\medskip

Note that, even if this assumption is quite general, there exist star-shaped objects whose convex hulls are not convex polytopes.
For an example, consider \cref{ex:higher_spiral}.

We denote by $\{x_j\}_{j=0}^{m}$ the vertices of $\ch$. 
Fix a center $c$ of $M$. 
Let $\{\sigma_i\}_{i=0}^{n}$ be the collection of convex polytopes on the boundary of $\ch$. 
For each polytope $\sigma_i$, define $S_i$ to be the convex hull of $\sigma_i\cup \{c\}$. 
The \define{sector} $M^i$ in $M$ is the intersection $S_i\cap M$.  
With a little abuse of notation, we use the word sector to refer both to $S_i$ and $M^i$.
See \Cref{fig:sector_notation} for an illustration of our notation. 
Note that each $x_j$ as well as all the points on the segment $x_jc$ belong to multiple sectors. 

In $\R^2$, the number of vertices and sectors of $\ch$ coincide. 
Thus, we can enumerate them so that the $i$-th sector has vertices $x_i$, $x_{(i+1) \mod (m+1)}$, and $c$. 
We will continue to use this notation in the next section.
\medskip

\begin{figure}[h]
\centering
\begin{subfigure}[t]{0.5\textwidth}
    \centering
    \includegraphics[width = .8\textwidth]{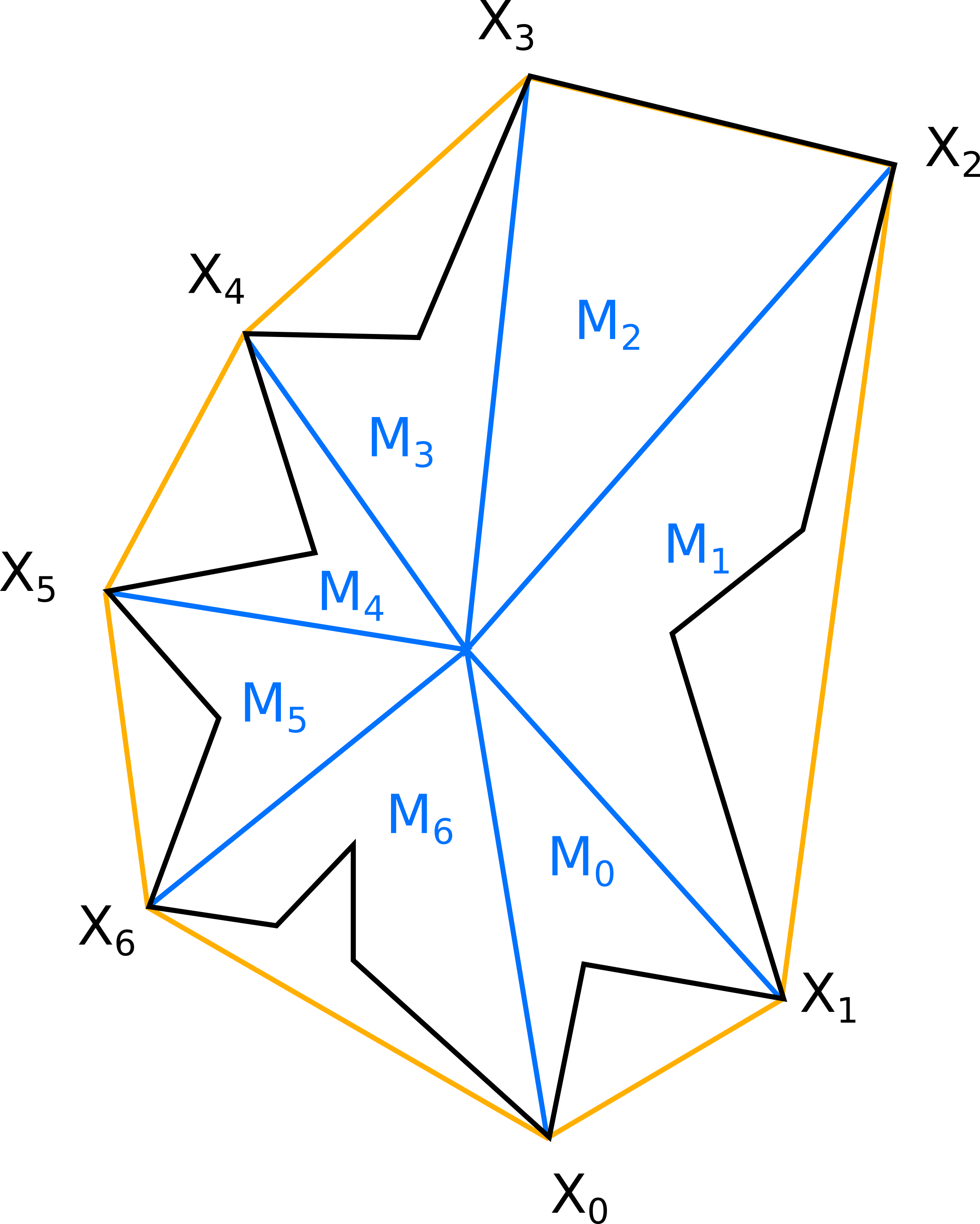}
    \caption{A star-shaped object $M$.}
    \label{fig:sector_notation}
\end{subfigure}%
\begin{subfigure}[t]{0.5\textwidth}
    \centering
    \includegraphics[width = .9\textwidth]{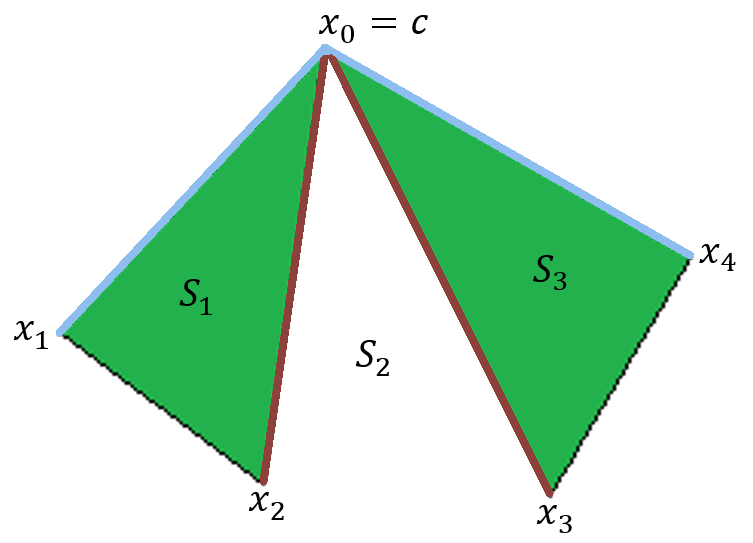}
    \caption{A star-shaped object $N$.}
    \label{fig:holed_starshaped}
\end{subfigure} 
\caption{Two star-shaped objects, $M$ (left) and $N$ (right), divided into sectors. (a) We show the boundary of $M$ in black and the boundary of its convex hull $\ch$ in orange. Each sector $M^i$ is the intersection of $M$ with $S_i$, which is the triangle with vertices $\{x_i, x_{(i+1) \mod (m+1)}, c\}$, where $m +1$ is the number of vertices on the convex hull and $c$ is a center of $M$. (b) The shape $N$ is the union of the two green triangles. Its convex hull is the pentagon with vertices $x_0 = c$, $x_1, \ldots, x_4$. We highlight the fact that there are five sectors, even though $N$ consists of only two triangles. Sectors $M^0$ and $M^4$ consist of the line segments $x_1c$ and $x_4c$, respectively, and are thickened in light blue. Sector $M^2$ consists of the two line segments $x_2c$ and $x_3c$, thickened in brown.}
\end{figure}

In analogy with \Cref{eq:sublevel}, we define $M^i_{\V, t}$ to be the $t$-sublevel set of $h_{\V}\vert_{M^i}$, where $h_\V$ is the height function defined in \Cref{ex: height function}. 
In other words,
\begin{equation*}
    M^i_{\V,t} \coloneqq \{x \in M^i \mid h_\V(x) \leq t\} = M^i \cap M_{\V,t}\,.
\end{equation*}
When $\V$ is clear from context, we omit it from our notation and write $M_t$ and $M^i_t$ instead of $M_{\V, t}$ and $M^i_{\V, t}$, respectively.

\section{Decomposition of the PHT}\label{sec:alg_dec}

The goal of this section is to prove that, for planar star-shaped objects, the non-essential classes correspond to ``local'' features. 
More precisely, we show that the reduced PHT of $M\subset \R^2$, star-shaped, is entirely determined by the reduced PHT of each of its sectors. 
For a fixed $\V \in \bS^1$, each non-essential persistent homology class of $\Dgm_0(M, \V)$ corresponds to a non-essential persistent homology class of $\Dgm_0(M^i, \V)$ for some sector $M^i$.
\medskip

Throughout this section, fix $\V \in \bS^1$. 
Recall that $c$, $x_i$, and $x_{(i+1) \mod (m+1)}$ are the vertices of the $i$-th sector $S_i$, where $x_0, \ldots, x_m$ are the vertices of the convex hull of $M$.
Without loss of generality, we may enumerate the vertices such that $\{M^i\}_{i=0}^{\ell}$ is the set of sectors  $M^i$ such that either $h_\V(x_i) < h_\V(c)$ or $h_\V(x_{(i+1) \mod (m+1)}) < h_\V(c)$. In other words, the set $\{M^i\}_{i=0}^\ell$ consists of the sectors such that $M^i_t \neq \emptyset$ for some $t < h_\V(c)$. In particular, for sectors $M^i$ with $i > \ell$, we have $M^i_t = \emptyset$ for all $t < h_\V(c)$.

Define the inclusions
\begin{align*}
    \iota^{i, t} &\colon M^i_t \xhookrightarrow{} \bigcup_{j = i}^\ell M^j_t \\
    \eta^{i, t} &\colon \bigcup_{j = i+1}^\ell M^j_t \xhookrightarrow{} \bigcup_{j = i}^\ell M^j_t \, .
\end{align*}
Let $\iota^{i, t}_*\colon \widetilde{H_0}(M^i_t) \to \widetilde{H_0}\Big(\bigcup_{j = i}^\ell M^j_t\Big)$ and $\eta^{i, t}_*\colon\widetilde{H_0}\Big(\bigcup_{j = i+1}^\ell M^j_t\Big) \to \widetilde{H_0}\Big(\bigcup_{j = i}^\ell M^j_t\Big)$ be the maps on reduced homology that are induced by inclusion. 
With an abuse of notation, for each $i$, these inclusions also define a map
\begin{equation}\label{eq:MV_isos}
    \widetilde{H_0}(M^i_t) \oplus \widetilde{H_0}\Big(\cup_{j=i+1}^\ell M_t^j\Big)\xrightarrow[]{\iota^{i, t}_* - \eta^{i, t}_*} \widetilde{H_0}\Big(\cup_{j=i}^\ell M_t^j\Big)\,.
\end{equation}
where $\iota^{i, t}_* - \eta^{i, t}_* (x,y)= \iota^{i, t}_*(x)- \eta^{i, t}_*(y)$.

We use the maps in \Cref{eq:MV_isos} to define a map
\begin{equation*}
    \alpha_t : \bigoplus_{i=0}^m \widetilde{H_0}(M^i_t) \to \widetilde{H_0}(M_t)\,.
\end{equation*}
For $t < h_\V(c)$, let $\alpha_t$ be the composition of maps
\begin{align}
    \bigoplus_{i=0}^m \widetilde{H_0}(M^i_t) &\xrightarrow[]{\iota^{\ell-1, t}_* - \eta^{\ell-1, t}_*} \bigoplus_{i=0}^{\ell-2} \widetilde{H_0}(M^i_t) \oplus \widetilde{H_0}(M^{\ell-1}_t \cup M^\ell_t) \oplus \bigoplus_{i=\ell +1}^m \widetilde{H_0}(M^i_t) \notag \\
    &\xrightarrow[]{\iota^{\ell-2, t}_* - \eta^{\ell-2, t}_*} \bigoplus_{i=0}^{\ell-3} \widetilde{H_0}(M^i_t) \oplus \widetilde{H_0}\Big( M^{\ell-2}_t \cup M^{\ell-1}_t \cup M^\ell_t\Big) \oplus \bigoplus_{i=\ell +1}^m \widetilde{H_0}(M^i_t) \notag \\
    \cdots &\xrightarrow[]{\iota^{0, t}_* - \eta_*^{0, t}} \widetilde{H_0}\Big(\bigcup_{i=0}^\ell M_t^i\Big) \oplus \bigoplus_{i=\ell +1}^m \widetilde{H_0}(M^i_t) \cong \widetilde{H_0}(M_t)\,. \label{eq:alpha}
\end{align}
The last isomorphism follows from the fact that if $t < h_\V(c)$, then $M^i_t = \emptyset$ for $i > \ell$, so $M_t = \bigcup_{i=0}^\ell M^i_t$. For $t \geq h_\V(c)$, define $\alpha_t \equiv 0$.

\begin{lemma}\label{lem:H1_zero}
    Let $M\subset \R^2$ be a star-shaped object, $\V\in\bS^1$, and $c$ a center of $M$.
    Then $H_1(\bigcup_{j=i}^\ell M^j_t) = 0$ for each $i$ and $t < h_\V(c)$.
\end{lemma}

\begin{proof}
Suppose that $x \in M^j_t$, and let $L_x$ be the line segment from $x$ to $c$, which is contained in $M^j$ by definition of $c$ and $M^j$. 
Let $L_{xt} \coloneqq L_x \cap M^j_t$. 
The other endpoint of $L_{xt}$ (besides $x$ itself) must be in the $t$-level set $h_\V^{-1}(t)$ because $t < h_\V(c)$, so $c \not\in L_{xt}$. 
We deformation retract each $x \in \bigcup_{j=1}^\ell M^j_t$ along the line $L_{xt}$ to obtain a deformation retraction to $h_\V^{-1}(t)$.
Therefore, $H_1(\bigcup_{j=i}^\ell M^j_t) = H_1(h_\V^{-1}(t))$. 
The $t$-level set $h_\V^{-1}(t)$ is the intersection of a line with $\bigcup_{j=1}^\ell M^j_t$, so $h_\V^{-1}(t)$ is a set of disjoint line segments. 
Each line segment is contractible, so $H_1(h_\V^{-1}(t)) = 0$.
\end{proof}

\begin{lemma}\label{lem:MV}
    Let $M\subset \R^2$ be a star-shaped object, $\V\in\bS^1$, and $c$ a center of $M$.
    For all $i \in \{0, \ldots, \ell - 1\}$ and $t < h_\V(c)$, the map
    \begin{equation*}
        \widetilde{H_0}(M^i_t) \bigoplus \widetilde{H_0}\Big(\bigcup_{j = i+1}^\ell M^j_t\Big) \xrightarrow[]{\iota^{i, t}_* - \eta^{i, t}_*} \widetilde{H_0}\Big( \bigcup_{j = i}^\ell M^j_t\Big)
    \end{equation*}
    is an isomorphism.
\end{lemma}
\begin{proof}
    From our star-shaped assumption and the ordering of the sectors (see discussion after \cref{ex_convex}), we know that $M^i_t \cap \bigcup_{j = i+1}^\ell M^j_t$ is non-empty.
    By \Cref{lem:H1_zero} and the reduced Mayer-Vietoris sequence \cite[Section 2.2, Mayer-Vietoris Sequences]{hatcher}, we have a short exact sequence
    \begin{equation*}
        0 \xrightarrow[]{} \widetilde{H_0}\Big(M^i_t \cap \bigcup_{j = i+1}^\ell M^j_t\Big) \to \widetilde{H_0}(M^i_t) \oplus \widetilde{H_0}\Big(\bigcup_{j = i+1}^\ell M^j_t\Big) \xrightarrow[]{\iota_*^{i,t} - \eta_*^{i,t}}\widetilde{H_0}\Big(\bigcup_{j=i}^\ell M^j_t\Big) \to 0\,.
    \end{equation*}
    Therefore, $\iota_*^{i,t} - \eta_*^{i,t}$ induces an isomorphism
    \begin{equation}\label{eq:quotient_iso}
        \widetilde{H_0}(M^i_t) \oplus \widetilde{H_0}\Big(\bigcup_{j = i+1}^\ell M^j_t\Big)/\widetilde{H_0}\Big(M^i_t \cap \bigcup_{j = i+1}^\ell M^j_t\Big)  \cong \widetilde{H_0}\Big(\bigcup_{j=i}^\ell M^j_t\Big)\,.
    \end{equation}
    For all $t$, the sector $M^0_t$ is only adjacent to $M^1_t$, the sector $M^\ell_t$ is only adjacent to $M^{\ell - 1}_t$, and for $i \in \{1, \ldots, \ell-1\}$, the sector $M^i_t$ is only adjacent to $M^{i-1}_t$ and $M^{i+1}_t$. Therefore,
    \begin{equation}\label{eq_R2_needed}
        M^i_t \cap \bigcup_{j = i+1}^\ell M^j_t = M^i_t \cap M^{i+1}_t\,,
    \end{equation}
    which is a line segment, hence connected, so $\widetilde{H_0}\Big(M^i_t \cap \bigcup_{j = i+1}^\ell M^j_t\Big) = 0$. 
    By \Cref{eq:quotient_iso}, the proof is complete.
\end{proof}

\begin{lemma}\label{lem:alpha_iso}
    Let $M\subset \R^2$ be a star-shaped object, $\V\in\bS^1$, and $c$ a center of $M$.
    For $t < h_\V(c)$, define $\alpha_t$ as in \Cref{eq:alpha}. For $t \geq h_\V(c)$, let $\alpha_t \equiv 0$.
    Then, for all $t\in\R$, the map $\alpha_t$ is an isomorphism.
\end{lemma}
\begin{proof}
For $t < h_\V(c)$, the map $\alpha_t$ is a composition of isomorphisms by \Cref{lem:MV}. 
For $t \geq h_\V(c)$, the set $M_t$ is connected and the sets $M^i_t$ are connected for each $t$, so $\widetilde{H_0}(M_t) = 0$ and $\widetilde{H_0}(M^i_t) = 0$ for each $i$.
\end{proof}

\begin{lemma}\label{lem:inclusions_commute}
    Let $X$ be a topological space with filtration $\{X_t\}$. 
    Let $A \subset X$ be a subspace and define $A_t \coloneqq X_t \cap A$. 
    If we define the inclusion maps
    \begin{align*}
        i^t &\colon A_t \hookrightarrow X_t \\
        j^{st} &\colon A_s \hookrightarrow A_t \\
        k^{st} &\colon X_s \hookrightarrow X_t \,,
    \end{align*}
    then the following diagram commutes for all $s \leq t$:
    \begin{equation*}
    \begin{tikzcd}
        \widetilde{H_0}(A_s) \arrow[r, "i^s_*"] \arrow[d, "j^{st}_*"'] & \widetilde{H_0}(X_s) \arrow[d, "k^{st}_*"] \\
        \widetilde{H_0}(A_t) \arrow[r, "i^t_*"']                       & \widetilde{H_0}(X_t)             
    \end{tikzcd}
    \end{equation*}
\end{lemma}

\begin{proof}
    The functions $i^t \circ j^{st}$ and $k^{st} \circ i^s$ are both equal to the inclusion $A_s \xhookrightarrow{} X_t$, so $i_t^* \circ j^{st}_* = k_*^{st} \circ i_*^s$.
\end{proof}

\begin{lemma}\label{lem:commute}
    Let $M\subset \R^2$ be a star-shaped object, $\V\in\bS^1$, and $c$ a center of $M$.
    For $t < h_\V(c)$, define $\alpha_t$ as in \Cref{eq:alpha}. For $t \geq h_\V(c)$, let $\alpha_t \equiv 0$. Then the following diagram commutes for all $s \leq t$:
    \begin{equation*}
        \begin{tikzcd}
            \displaystyle\bigoplus_{i=0}^m\widetilde{H_0}(M_s^i) \arrow[r, "\alpha_s"] \arrow[d, "\oplus_i \psi_{st}^i"'] & \widetilde{H_0}(M_s) \arrow[d, "\phi_{st}"] \\
            \displaystyle\bigoplus_{i=0}^m\widetilde{H_0}(M_t^i) \arrow[r, "\alpha_t"']                                   & \widetilde{H_0}(M_t)\,,                      
        \end{tikzcd}
    \end{equation*}
    where $\psi^i_{st}\colon \widetilde{H_0}(M^i_s) \to \widetilde{H_0}(M^i_t)$ is the map induced by the inclusion $M^i_s \xhookrightarrow{} M^i_t$ and $\phi_{st}\colon \widetilde{H_0}(M_s) \to \widetilde{H_0}(M_t)$ is the map induced by the inclusion $M_s \xhookrightarrow{} M_t$.
\end{lemma}

\begin{proof}
    \underline{Case 1: ($t \geq h_\V(c)$)}: If $t \geq h_\V(c)$, then $M_t$ is connected so $\widetilde{H_0}(M_t) = 0$. Therefore $\alpha_t \circ \oplus_i \psi_{st}^i = 0 = \phi_{st} \circ \alpha_s$.

    \vspace{3mm}

    \underline{Case 2: ($t < h_\V(c)$)}: By definition of $\alpha_t$ as the composition of maps $\iota_*^{i, t} - \eta_*^{i, t}$, it suffices to show that for $i \in \{0, \ldots, \ell-1\}$, the following diagram commutes:

    \begin{equation}\label{eq:reduced_diagram}
        \begin{tikzcd}
\widetilde{H_0}(M^i_s) \oplus \widetilde{H_0}\Big(\bigcup_{j = i+1}^\ell M^j_s\Big) \arrow[d, "\psi^i_{st} \oplus \xi^{i+1}_{st}"'] \arrow[r, "\iota_*^{is} - \eta_*^{is}"] & \widetilde{H_0}\Big(\bigcup_{j=i}^\ell M^j_s\Big) \arrow[d, "\xi^i_{st}"] \\
\widetilde{H_0}(M^i_t) \oplus \widetilde{H_0}\Big(\bigcup_{j = i+1}^\ell M^j_t\Big) \arrow[r, "\iota_*^{it} - \eta_*^{it}"']                                                & \widetilde{H_0}\Big(\bigcup_{j=i}^\ell M^j_t\Big)                        
\end{tikzcd}
    \end{equation}
    where $\xi^i_{st}\colon \widetilde{H_0}\Big(\bigcup_{j = i}^\ell M^j_s\Big) \to \widetilde{H_0}\Big(\bigcup_{j = i}^\ell M^j_t\Big)$ is induced by the inclusion $\bigcup_{j=i}^\ell M^j_s \hookrightarrow \bigcup_{j=i}^\ell M^j_t$. 
    To see why this is sufficient, note that $\xi^{\ell}_{st} = \psi^\ell_{st}$ and $\xi^0_{st} = \phi^{st}$.

    By \Cref{lem:inclusions_commute}, the following diagrams commute:
    \begin{equation*}
        \begin{tikzcd}
\widetilde{H_0}(M^i_s)  \arrow[d, "\psi^i_{st}"'] \arrow[r, "\iota_*^{i,s}"] & \widetilde{H_0}\Big(\bigcup_{j=i}^\ell M^j_s\Big) \arrow[d, "\xi^i_{st}"] &  \widetilde{H_0}\Big(\bigcup_{j = i+1}^\ell M^j_s\Big) \arrow[d, "\xi^{i+1}_{st}"'] \arrow[r, "\eta_*^{is}"] & \widetilde{H_0}\Big(\bigcup_{j=i}^\ell M^j_s\Big) \arrow[d, "\xi^i_{st}"] \\
\widetilde{H_0}(M^i_t) \arrow[r, "\iota_*^{it}"']                           & \widetilde{H_0}\Big(\bigcup_{j=i}^\ell M^j_t\Big)                         &  \widetilde{H_0}\Big(\bigcup_{j = i+1}^\ell M^j_t\Big) \arrow[r, "\eta_*^{i,t}"]                              & \widetilde{H_0}\Big(\bigcup_{j=i}^\ell M^j_t\Big)\,,        
\end{tikzcd}
\end{equation*}
which implies that the diagram in \Cref{eq:reduced_diagram} commutes, completing the proof.
\end{proof}

\begin{theorem}\label{thm:sector_decomposition}
    Let $M\subset \R^2$ be a star-shaped object and denote its sectors by $\{M^i\}_{i=0}^m$. 
    Then, for all $\V \in \bS^1$,
    \[
    \bigoplus_{i=0}^m \widetilde{\text{PH}_0}(M^i, \V)\cong \widetilde{\text{PH}_0}(M, \V) \, .
    \]
\end{theorem}

\begin{proof}
    Let $\alpha_t$ be defined as in \Cref{eq:alpha} for $t < h_\V(c)$, with $\alpha_t \equiv 0$ for $t \geq h_\V(c)$. By \Cref{lem:alpha_iso}, $\alpha_t$ is an isomorphism of vector spaces at each $t$. By \Cref{lem:commute}, it commutes with the structure maps of each persistence module, so it is an isomorphism of persistence modules.
\end{proof} 

In order to prove \cref{thm:sector_decomposition}, we crucially needed that, in \cref{eq_R2_needed}, the intersections are connected. 
This is in general not true for star-shaped objects in $\R^d$ for $d>2$, where a sector may intersect with more than two other sectors.
This is not a shortcoming of our approach; rather, as we show in \cref{no_higherdim_LV}, \cref{thm:sector_decomposition} does not hold in higher ambient dimension without additional assumptions.

Nevertheless, there are star-shaped objects, in arbitrary ambient dimensions, for which \cref{thm:sector_decomposition} does hold. 
Since in what follows we need the decomposition of the $\Dgm_0$ into sectors, but not necessarily the restriction to an ambient dimension of $2$, we provide the following definition.

\begin{definition} \label{def:sectorial}
A star-shaped object $M$ in $\R^d$ is called \define{sectorial} if 
\[
\bigoplus_{i=0}^n \widetilde{\text{PH}_0}(M^i, \V)\cong\widetilde{\text{PH}_0}(M, \V)
\]
for all $\V \in \bS^1$, where $\{M^i\}_{i=0}^n$ are the sectors of $M$.
\end{definition}

\begin{figure}[h!]
    \centering
    \begin{subfigure}[t]{1\textwidth}
    \centering
    \includegraphics[scale=.12]{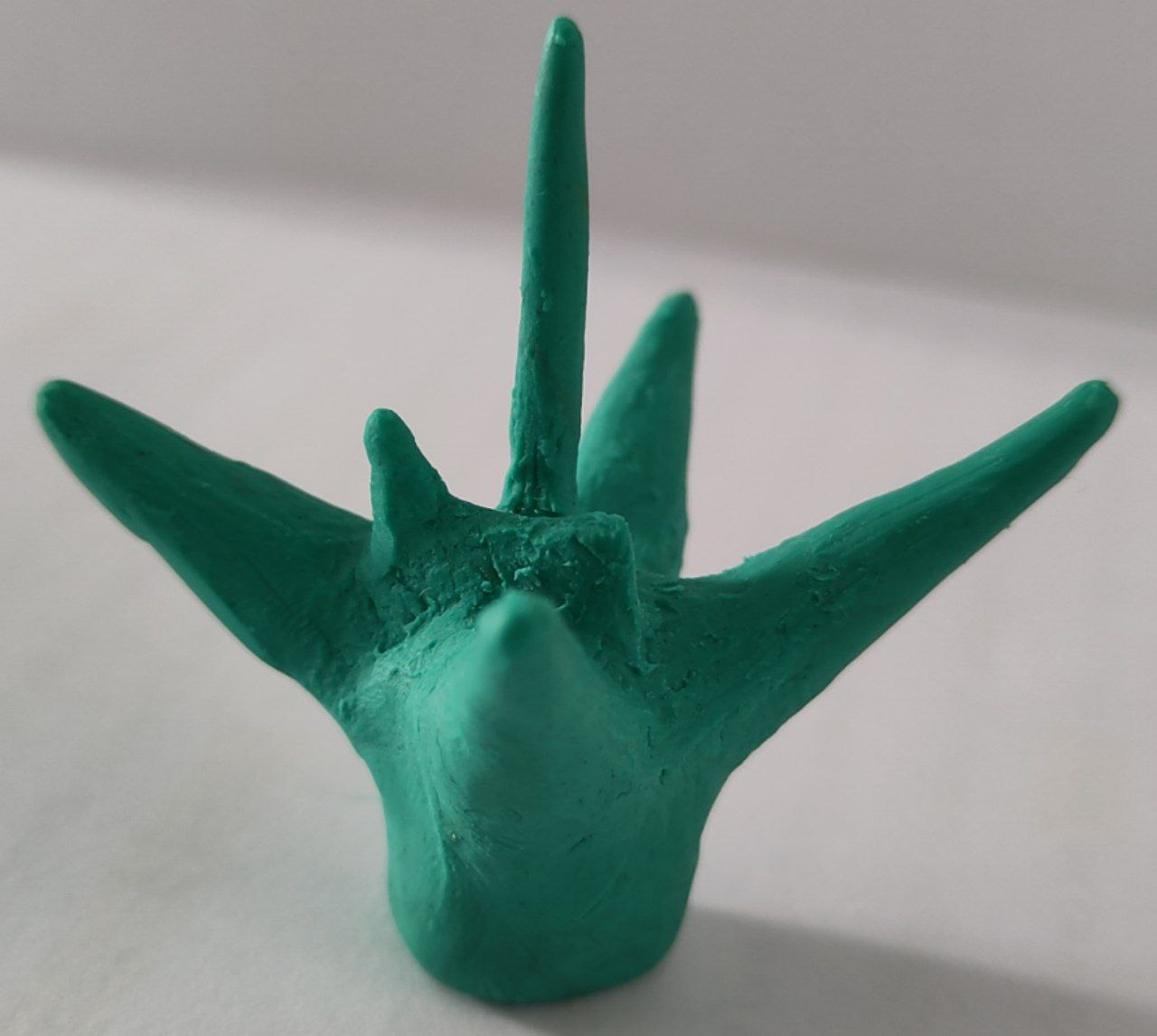}
\end{subfigure}%
    \caption{3D object in modeling clay showing that~\cref{thm:sector_decomposition} does not fully generalize in $\R^3$. Its contours in the vertical direction (from top to bottom) are described in \cref{no_higherdim_LV} and depicted in \cref{fig:no_higher_LV}.}
    \label{fig:geom_realization_noLV}
\end{figure}

\begin{figure}[!h]
\begin{minipage}{0.49\textwidth}
\centering
  \begin{subfigure}[t]{.9\textwidth}
    \centering
    \includegraphics[scale=.32]{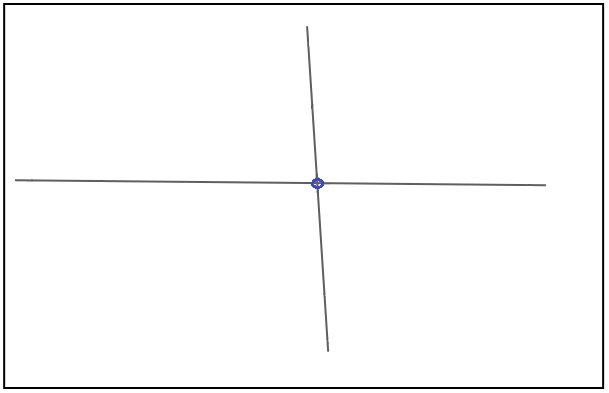}
    \caption{}
  \end{subfigure}
  \begin{subfigure}[t]{.9\textwidth}
    \centering
    \includegraphics[scale=.32]{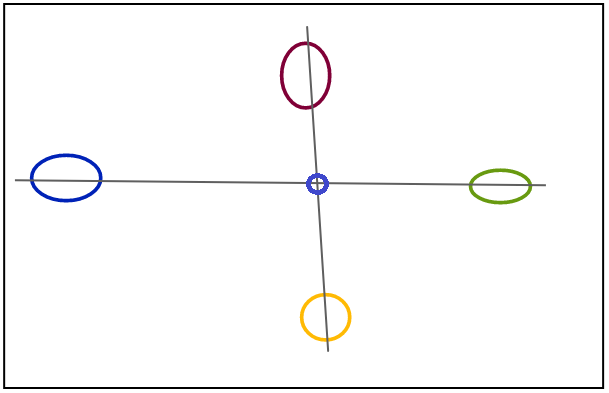}
    \caption{}
  \end{subfigure}
  \begin{subfigure}[t]{.9\textwidth}
    \centering
    \includegraphics[scale=.32]{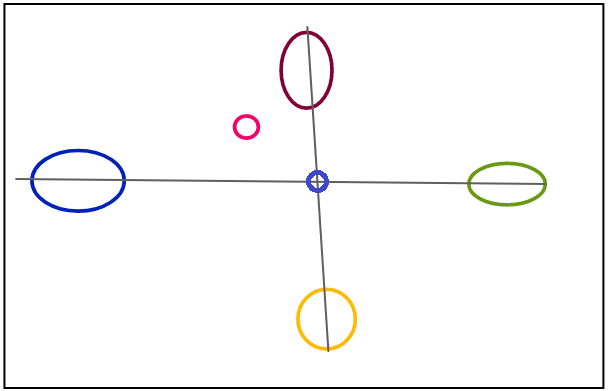}
    \caption{}
  \end{subfigure}
\end{minipage}
\begin{minipage}{0.49\textwidth}
\centering
  \begin{subfigure}[t]{.9\textwidth}
    \centering
    \includegraphics[scale=.32]{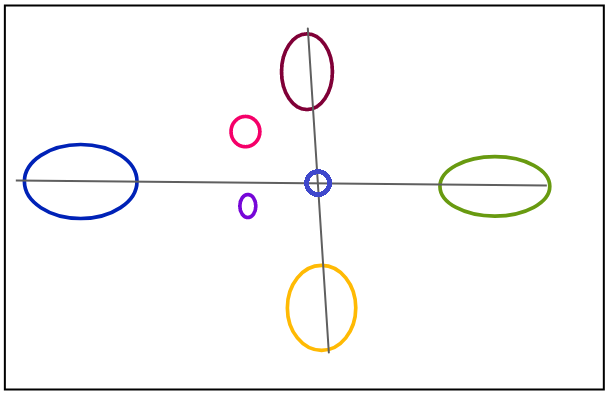}
    \caption{}
  \end{subfigure}
  \begin{subfigure}[t]{.9\textwidth}
    \centering
    \includegraphics[scale=.32]{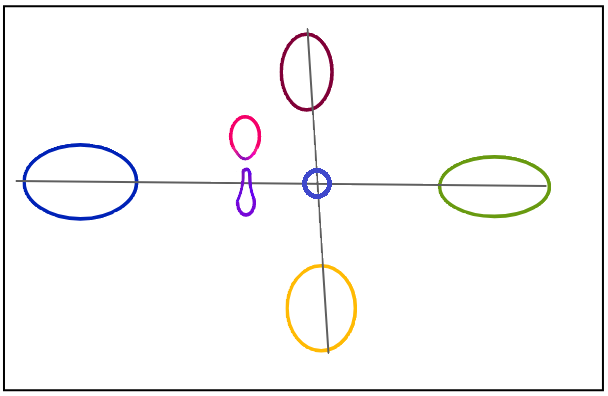}
    \caption{}
  \end{subfigure}
  \begin{subfigure}[t]{.9\textwidth}
    \centering
    \includegraphics[scale=.32]{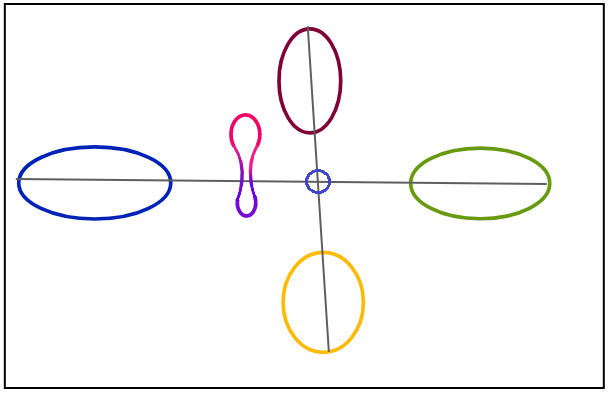}
    \caption{}
  \end{subfigure}
\end{minipage}
\caption{The levelset contours of the star-shaped object $M \subset \R^3$ that is shown in~\cref{fig:geom_realization_noLV}, demonstrating that \cref{thm:sector_decomposition} does not, in general, hold in higher dimensions. In (a)--(f), we show the $t$-levelset contours of $h_\V \colon M \to \R$ for increasing $t$, where $\V = (0, 0, -1)$. The gray lines represent the boundaries of the sectors. }
    \label{fig:no_higher_LV}
\end{figure}

\begin{example}\label{no_higherdim_LV}
\cref{fig:no_higher_LV} illustrates the contours of the sublevel sets of the star-shaped object in $\R^3$ depicted in \cref{fig:geom_realization_noLV}, where the direction for the height function is pointing downward. 
The straight lines in \cref{fig:no_higher_LV} represent the division of sectors.
In \cref{fig:geom_realization_noLV}, these straight lines are not drawn but would go from the highest horn in the middle to each of the four most external horns.
At the beginning, we have one connected component, given by the highest vertical horn in \cref{fig:geom_realization_noLV}.
Then four other connected components, given by the four most external horns in \cref{fig:geom_realization_noLV}, enter at the same time. 
Then we have the birth of a sixth component (pink) inside a sector, followed by the birth of a seventh component (purple) in another sector. 
There are given by the two smaller horns in \cref{fig:geom_realization_noLV}.
These two components merge, but the merging point is not in the same sector as the birth point of the seventh-born purple component, so $\widetilde{\Dgm}_0(M, \V)$ has a point that is not contained in $\widetilde{\Dgm}_0(M^i, \V)$ for any sector $M^i$.
\end{example}

\section{The PHTs of star shapes have trivial geometric monodromy}\label{sec:monodromy}

In this section, we investigate a class of objects whose PHTs have trivial geometric monodromy. 
By contrast, recall from \Cref{sec:monodromy_def} that in general, the PHT of a shape $M$ need not have trivial monodromy; see \Cref{fig:spiral} for an example. 
Our main results (\Cref{prop:no_monodromy_suff_conditions} and \Cref{thm:no_monodromy}) assume that $M \subset \R^2$, but some of our proofs generalize straightforwardly to arbitrary dimension $d$, so we present those in generality.

We first show (\Cref{prop:no_monodromy_suff_conditions}) that if $M$ is a star-shaped object in $\mathbb{R}^2$ such that its persistence diagrams are ``simple'' (\Cref{def:simple_PD}) then its PHT has trivial monodromy.
We discuss an obstruction in higher dimensions in \cref{ex:higher_spiral}.  

We then conclude by showing that if $M$ is a star-shaped polygon in $\R^2$ whose vertices are in general position, then $M$ has simple persistence diagrams. Consequently, the PHT of $M$ has trivial geometric monodromy.
\medskip 

Recall the sectors $\{M^i\}$ of a star-shaped object $M\subset\R^d$ as defined in \Cref{sec:star_shaped_decomposition} and that all star-shaped objects we consider here have as a convex hull a convex polytope.

\begin{lemma}\label{lem:finite_support_Mi}
Let $M$ be a star-shaped object in $\R^d$ and let $c$ be a center of $M$.
For each sector $M^i$, there exists a direction $\V\in \bS^{d-1}$ such that
\begin{equation*}
    \Dgm_0(M^i, \V) = \{(h_{\V}(c), \infty)\} \cup \Delta \, .
\end{equation*}
\end{lemma}

\noindent In other words, there exists a direction for which the reduced degree-$0$ persistent homology of a sector is trivial.

\begin{proof}
Let $\sigma_i$ be the polytope on $\ch$ defining the sector $M^i$. 
\vspace{3mm}

\noindent \underline{Case 1:} If $c$ belongs to $\sigma_i$, then $M^i = \sigma_i$. Let $\V$ be a unit vector that is orthogonal to the hyperplane $P$ containing $\sigma_i$. Then $h_{\V}(z) = h_{\V}(c)$ for all $z \in M^i$.
\vspace{3mm}

\noindent \underline{Case 2:} If $c$ does not belong to $\sigma_i$, consider the line $L_{c\sigma_i}$ passing through $c$ and perpendicular to $\sigma_i$. 
Let $y\coloneqq L_{c\sigma_i}\cap\sigma_i$.
Let $\V = (y - c)/\norm{y - c}$. 
For all points $z \in M^i$, the angle $\theta$ between $\V$ and $z - c$ is at most $\pi/2$. 
Consequently,
\begin{equation*}
    h_\V(z) - h_\V(c) = \ip{\V}{(z - c)} = \norm{\V}\norm{z - c} \cos(\theta) \geq 0
\end{equation*}
for all points $z\in M^i$.

\vspace{3mm}
In either case, $M^i_t = \emptyset$ for $t < h_\V(c)$ and $M^i_t$ is connected for all $t \geq h_\V(c)$, thus proving our claim.
\end{proof}

In order to apply the definition of trivial geometric monodromy we will further need to assume that the PHT of our shape $M$ is a simple persistence diagram bundle (\Cref{def:simple_pdb}).

\begin{definition}\label{def:simple_PD}
    An object $M$ in $\R^d$ has \define{simple $\Dgm_0$} if, for all $\V \in \bS^{d-1}$, the persistence diagram $\Dgm_0(M, \V)$ has no points of multiplicity $> 1$ and has only finitely many points.
\end{definition}

In the following results, we study the total space $E \subset \bS^{d-1} \times (\R^2\cup \Delta)$ (see \Cref{def:pdb}) of the reduced persistent homology transform. 
Our goal is to understand the connected components of $E \setminus (\bS^{d-1} \times \Delta )$.

\begin{lemma}\label{lem:vine_within_sector}
Assume that $M\subset \R^d$ is star-shaped, sectorial, and has simple $\Dgm_0$.
If $\gamma$ is a connected component in the set
    \begin{equation*}
        \{(\V, z) \mid \V \in \bS^{d-1}\, , \, z \in \widetilde{\Dgm}_0(M, \V) \setminus \Delta \}\,,
    \end{equation*}
    then there is an $i$ such that
    \begin{equation*}
        \gamma \subset \{(\V, z) \mid \V \in \bS^{d-1}\, , \, z \in \widetilde{\Dgm}_0(M^i, \V) \setminus \Delta \}\,.
    \end{equation*}
\end{lemma}

We use the standard notation $B(x,r)$ for an open ball centered at $x$ of radius $r$ in whichever metric the context indicates.

\begin{proof}
    The shape $M$ must be bounded because its convex hull is bounded.
    Fix $R\geq 1$ such that $M \subset B(0,R)$. 

    For each $j$, define the sets
    \begin{align*}
        E_j &\coloneqq \{(\V, z) \mid \V \in \bS^{d-1},\,  \, z \in \widetilde{\Dgm}_0(M^j, \V)  \setminus \Delta\}\,, 
        \\
        \mathcal{U}_j &\coloneqq \gamma \cap E_j\,.
    \end{align*}
    Because $M$ is sectorial, 
    $\gamma = \cup_j \mathcal{U}_j$. The sets $\{E_j\}$ are disjoint because $M$ has simple $\Dgm_0$, so the sets $\{\mathcal{U}_j\}$ are disjoint, too.
    We wish to show that $\mathcal{U}_i$ is open in $\gamma$ for all $i$. 
    This is trivially true for empty $\mathcal{U}_i$. 
    Fix $i$ such that $\mathcal{U}_i$ is non-empty.

    It suffices to show that for all $(\V, z) \in \mathcal{U}_i$, there exists a $\delta>0$ such that 
    \[
    B((\V, z), \delta)\cap \gamma \subset \mathcal{U}_i \,,
    \]
    where 
    \begin{equation*}
        B((\V, z), \delta) := \{ (\mathbf{u}, z') \in \bS^{d-1} \times \R^2 \mid \norm{(\mathbf{u}, z') - (\V, z)}_2 < \delta\}
    \end{equation*}
    is an open ball of radius $\delta$ with respect to the Euclidean norm, with $\bS^{d-1} \subset \R^d$.
    Let $(\V, z)\in\mathcal{U}_i$ and set
    \begin{equation*}
        \epsilon \coloneqq \min_{j \neq i} \inf_{y\in \widetilde{\Dgm}_0(M^j, \V)} \norm{z- y}_2 \, .
    \end{equation*}
    Note that $\widetilde{\Dgm}_0(M^j, \V) \cap \widetilde{\Dgm}_0(M^i, \V) = \Delta$ for all $i \neq j$ because $M$ has simple $\Dgm_0$, so we must have $\epsilon > 0$.

    Suppose that $({\bf u}, z')\in B((\V, z), \frac{\epsilon}{2R\sqrt{2}})\cap \gamma$.
   
    For any $j \neq i$, let $w'$ be any point in $\widetilde{\Dgm}_0(M^j, u)$. We will show that $w' \neq z'$. The height functions $h_{\textbf{u}}$ and $h_\V$ corresponding to ${\bf u}$ and $\V$ are within $R\norm{{\bf u}-\V}_2$ sup-norm of each other. By bottleneck stability \cite{OPstability},
    \begin{equation*}
        d_B(\Dgm_0(M^j, \textbf{u}), \Dgm_0(M^j, \V)) \leq R \norm{{\bf u} - \V}_2 \leq \frac{\epsilon}{2 \sqrt{2}}\,.
    \end{equation*}
    Therefore, there is a $w \in \Dgm_0(M^j, \V)$ such that $\norm{w - w'}_{\infty} \leq \frac{\epsilon}{2 \sqrt{2}}$. This implies that
    \begin{equation*}
        \norm{z' - w'}_2 \geq \norm{z - w}_2 - \norm{z - z'}_2 - \norm{w - w'}_2 \geq \epsilon - \frac{\epsilon}{2R \sqrt{2}} - \frac{\epsilon}{2} > 0\,,
    \end{equation*}
    so $z' \not\in \widetilde{\Dgm}_0(M^j, u)$ for any $j \neq i$. Therefore, we must have $z' \in \widetilde{\Dgm}_0(M^i, {\bf u})$ and $B((\V, z), \frac{\epsilon}{2R \sqrt{2}}) \cap \gamma \subset \mathcal{U}_i$.

    We have thus shown that $\mathcal{U}_i$ is open in $\gamma$ for all $i$. 
    Because $\cup_i \mathcal{U}_i = \gamma$ is connected by assumption and $\mathcal{U}_i \cap \mathcal{U}_j = \emptyset$ for all $i \neq j$, there must be an $i$ such that $\mathcal{U}_j = \emptyset$ if $j \neq i$, which concludes the proof.
\end{proof}
\begin{lemma}\label{lem:finite_support}

Assume that $M$ is star-shaped, sectorial, and has simple $\Dgm_0$.
Let $\pi\colon\bS^{d-1} \times \overline{\R}^2 \to\bS^{d-1}$ be the projection map.
If $\gamma$ is a connected subset of the set
\begin{equation*}
    \{(\V, z) \mid \V \in\bS^{d-1}\, , \, z \in \widetilde{\Dgm}_0(M, \V) \setminus \Delta \}\,,
\end{equation*}
then $\pi(\gamma)$ is a strict subset of $\bS^{d-1}$.
\end{lemma}

\begin{proof}
    By \Cref{lem:vine_within_sector}, there is an $i$ such that
    \begin{equation*}
        \gamma \subset \{(\V, z) \mid \V \in\bS^{d-1}\,, z \in \widetilde{\Dgm}_0(M^i, \V) \setminus \Delta \}\,.
    \end{equation*}
    By \Cref{lem:finite_support_Mi}, there is a direction ${\bf w} \in\bS^{d-1}$ such that $\widetilde{\Dgm}_0(M^i, {\bf w}) = \Delta$.
    Therefore ${\bf w} \not \in \pi(\gamma)$.
\end{proof}

\begin{rem}
In particular, \Cref{thm:sector_decomposition} implies that the hypotheses of \Cref{lem:vine_within_sector} and \Cref{lem:finite_support} are satisfied when $M$ is a star-shaped object in $\R^2$ such that it has simple $\Dgm_0$ (recall that all star-shaped objects we consider here have as a convex hull a convex polytope).
\end{rem}
    
\begin{lemma}\label{we_have_vines}
     Suppose that $M\subset \R^d$ is a bounded object with simple $\Dgm_0$, and define
     \begin{equation*}
     E' \coloneqq \{(\V, z) \mid \V \in \bS^{d-1}, \, z \in \widetilde{\Dgm_0}(M, \V) \setminus \Delta\} \, .
     \end{equation*}
     Let $\pi\colon\bS^{d-1}\times \overline{\R}^2 \to \bS^{d-1}$ be the projection map. 
     Then $\pi$ restricted to $E'$ is a local homeomorphism. 
\end{lemma}
    
\begin{proof}
It is sufficient to show that for all $(\V, z)\in E'$, there exist $\epsilon_1, \epsilon_2 >0$ such that $\pi$ restricted to $E'\cap \Big(B^{\bS^{d-1}}(\V, \epsilon_1) \times B^{\R^2}(z, \epsilon_2)\Big)$ is a homeomorphism onto $B^{\bS^{d-1}}(\V, \epsilon_1)$, where
\begin{align*}
    B^{\bS^{d-1}}(\V, \epsilon_1) &\coloneqq \{u \in \bS^{d-1} \mid \norm{u - v}_2 < \epsilon_1\}\,, \\
    B^{\R^2}(z, \epsilon_2) &\coloneqq \{z' \in \R^2 \mid \norm{z' - z}_2 < \epsilon_2\}
\end{align*}
are open balls of radius $\epsilon_1$, $\epsilon_2$, respectively, with respect to the Euclidean norm.
   
$M$ is bounded by assumption, so fix $R\geq 1$ such that $M \subset B^{\R^d}(0,R)$. 
Define 
\begin{equation*}
\delta = \min_{\substack{q \in \widetilde{\Dgm}_0(M, \V) \\ q \neq z}} \norm{z - q}_2\, .
\end{equation*}
Note that $\delta > 0$ because $M$ has simple $\Dgm_0$ by assumption.
Set $\epsilon_1 =\frac{\delta}{2\sqrt{2}R}$ and $\epsilon_2 = \frac{\delta}{2}$. If ${\bf u} \in B^{\bS^{d-1}}(\V, \epsilon_1)$, then $\norm{h_{\bf u} - h_\V}_{\infty} \leq R \norm{u - v}_2 \leq \frac{\epsilon_1}{2 \sqrt{2}}$, so bottleneck stability implies
    \begin{equation}\label{eq:db_bound}
    d_B(\Dgm_0(M, {\bf u}), \Dgm_0(M, \V)) < \frac{\delta}{2 \sqrt{2}}\, .
    \end{equation}
Therefore, there is a matching $\mathcal{M} \subseteq \Dgm_0(M, {\bf u}) \times \Dgm_0(M, \V)$ such that every non-diagonal element of $\Dgm_0(M, {\bf u})$, $\Dgm_0(M, \V)$ appears in exactly one pair, and $\norm{x - y}_2  < \frac{\delta}{2}$ for all $(x, y) \in \mathcal{M}$. Let $q_u \in \Dgm_0(M, {\bf u})$ be the unique point such that $(z, q_u) \in \mathcal{M}$. We must have $q_u \neq \Delta$ because $\norm{q_u - z}_2 < \frac{\delta}{2} < \delta \leq \norm{z - \Delta}_2$; therefore, $({\bf u}, q_{\bf u}) \in E' \cap \Big(B^{\bS^{d-1}}(\V, \epsilon_1) \times B^{\R^2}(z, \epsilon_2)\Big)$, where $\epsilon_2 = \frac{\delta}{2}$.

If $q' \in \Dgm_0(M, {\bf u})$ and $q' \neq q$, then let $z' \in \Dgm_0(M, \V)$ a point such that $(z', q') \in \mathcal{M}$. We must have $z' \neq z$ because $z$ only appears in one pair, so $\norm{z - z'}_2 \geq \delta$ by definition of $\delta$. Therefore,
\begin{equation*}
    \norm{q' - z}_2 \geq \norm{z - z'}_2 - \norm{z' - q'}_2 > \norm{z - z'}_2 - \frac{\delta}{2} \geq \frac{\delta}{2}\,,
\end{equation*}
so $q' \not\in B^{\R^2}(z, \frac{\delta}{2}$. Therefore,
\begin{equation*}
    E' \cap \Big( B^{\bS^{d-1}}(\V, \epsilon_1) \times B^{\R^2}(z, \epsilon_2)\Big) = \{({\bf u}, q_{\bf u})\}_{{\bf u} \in B^{\bS^{d-1}}(\V, \epsilon_1)}\,.
\end{equation*}
The projection $\pi: ({\bf u}, q_{\bf u}) \mapsto {\bf u}$ is a homeomorphism by bottleneck stability.
\end{proof}

We are now ready to show a sufficient condition for trivial geometric monodromy. 
In what follows, the symbol $\partial M$ denotes the boundary of a subset $M\subset \R^d$. 

\begin{proposition}\label{prop:no_monodromy_suff_conditions}
    Assume that $M\subset \R^2$ is star-shaped and has a simple $\Dgm_0$. 
    Then $\PHT_0(M)$ has trivial geometric monodromy.
\end{proposition}

\begin{proof}
Let $\hat{\gamma_1}, \ldots, \hat{\gamma_n} \subset \bS^1\times\overline{\R}^2 $ be the connected components of 
\begin{equation*}
    E \coloneqq \{(p, \V) \mid \V \in \bS^1, p \in \widetilde{\Dgm}_0(M, \V)\setminus \Delta\}
\end{equation*} 
when written as a persistence diagram bundle.
We know each $\hat{\gamma_i}$ is a strict subset of $\bS^1$ by \Cref{lem:finite_support_Mi}.
By \cref{we_have_vines}, $\pi$ restricted to $\hat{\gamma_i}$ is a local homeomorphism. 
Together they imply that $|\hat{\gamma_i} \cap \pi^{-1}(\V)|$ is at most $1$ for any $\V$.

Define a parameterized curve $\gamma_i \colon \bS^1 \to \R^2$ that extends $\hat{\gamma_i}$ by setting $\gamma({\bf v})=\Delta$ for ${\bf v} \notin \pi(\hat{\gamma})$. 
By construction $E= \bigcup_i \hat{\gamma_i}$ so  $ \widetilde{\Dgm}_0(M, \V)= \bigcup_{i \mid \gamma_i(\V) \neq \Delta} \gamma_i(\V)$ for all $\V \in \bS^1$.

To complete the proof we just need to verify that the $\gamma_i$ are each continuous.
As $\pi$ restricted to $\hat{\gamma}_i$ is a local homeomorphism, we know that $\pi(\hat{\gamma_i})$ is an open subset of $\bS^1$ and the only locations of potential discontinuity are the vectors $\V\in \partial(\pi(\hat{\gamma_i}))$. 
By way of contradiction suppose that $\gamma_i$ is discontinuous at $\V \in \partial(\pi(\hat{\gamma}_i))$. 
Without loss of generality (by rotating $M$ if necessary) assume that $\V=(1,0)$ and that $(\cos(\theta), \sin(\theta))\in \pi(\hat{\gamma}_i)$ for all small positive $\theta$. 
Let $z= \lim_{\theta \downarrow 0}\gamma_i(\cos(\theta), \sin(\theta))$. 
If $z$ is a location on the diagonal in $\R^2$ then it will not cause a discontinuity so we know that $z$ must be a positive distance to the diagonal. 
As the persistence diagrams are continuous with respect to changes in direction we must have
$(\V,z)\in E$, but this would imply it $(\V, z)\in \hat{\gamma_i}$ and $\V\in \pi(\hat{\gamma_i})$ which creates a contradiction.
\end{proof}

The proof of \cref{prop:no_monodromy_suff_conditions} cannot be generalized to sectorial objects in $\R^d$ without additional hypotheses. 
Any non-trivial monodromy over persistence bundles over $\bS^1$ can only occur through the continuation of a section the entire way around $\bS^1$. 
In higher ambient dimensions we must also consider potential non-trivial monodromy along smaller contractible loops in $\bS^{d-1}$.
An illustrative example is as follows.

\begin{example}\label{ex:higher_spiral}
Let $X$ be the spiral in \cref{fig:spiral} but embedded in $\R^3$ in the plane with $z$ coordinate $-1$. 
We then set $M=\{tx \mid x\in X, t\in [0,1]\}$. 
Note this is star-shaped by construction with a center in the origin.
We can consider any small loop around the north pole (but not going through the north pole) and we will get the same monodromy behavior as in \cref{ex_monodromy}. 
\end{example}
 
The above results show that if $M$ is a star-shaped object in $\R^2$ that has simple $\Dgm_0$, then $\PHT_0(M)$ has trivial geometric monodromy. 
Next, we provide sufficient conditions so that $M$ has simple $\Dgm_0$.
The following lemmas are used towards \Cref{thm:no_monodromy}, which gives us sufficient conditions on $M$ to have simple $\Dgm_0$ and, consequently by \Cref{prop:no_monodromy_suff_conditions}, trivial geometric monodromy in $\PHT_0(M)$. 
\medskip

Let $(\partial M)_{\V, t} \coloneqq \{ x \in \partial M \mid h_\V (x) \leq t\}$. 
Equivalently, $(\partial M)_{\V, t}$ is the $t$-sublevel set of $h_\V \vert_{\partial M}$. 
We remark that $(\partial M)_{\V, t} $ does not denote the boundary of $M_{\V, t}$---the former is a subset of the latter.

\begin{lemma}\label{lem:def_retract}
    Let $\V \in \bS^{d-1}$. 
    If $t < h_\V(c)$, then $M_{\V, t}$ deformation retracts to $(\partial M)_{\V, t}$.
\end{lemma}

\begin{proof} 
The lemma is trivially true for $(\partial M)_{\V, t}$ empty, as then $M_{\V, t}$ is also empty.

When $(\partial M)_{\V, t} \neq \emptyset$, we define a retraction $r^t \colon M_{\V, t} \to (\partial M)_{\V, t}$ as follows. 
Take $x \in M_{\V, t}$. 
If, in particular, $x\in (\partial M)_{\V, t}$, then $r^t(x)\coloneqq x$.
Assume then $x\not\in (\partial M)_{\V, t}$, and let $L_x$ be the line that contains points $x$ and $c$. 
Set $r^t(x)$ to be the closest point on $L_x\cap (\partial M)_{\V, t}$ to $x$. 
Note that this set $L_x\cap (\partial M)_{\V, t}$ will usually be a singleton. 
The straight line homotopy between the retraction and the identity map on $M_{\V, t}$ is a deformation retraction. 
\end{proof}

Equivalently, we can describe the deformation retraction in the above proof as following the flow of the vector field pointing away from $c$, halting whenever we hit \medskip$(\partial M)_{\V, t}$.

The next lemma shows that, in order to study the PHT of $M$, it suffices to study the persistent homology of $\{(\partial M)_{\V, t}\}_{t \in \R}$ for each $\V \in \bS^1$.

\begin{lemma}\label{lem:reduce_to_boundary}
    Let $\V \in \bS^1$ and suppose that $s \leq t < h_\V(c)$. 
    Let $\psi_{st}\colon H_0((\partial M)_{\V, s}) \to H_0((\partial M)_{\V, t})$ and $\phi_{st}\colon H_0(M_{\V, s}) \to H_0(M_{\V, t})$ be the maps induced by the inclusions $(\partial M)_{\V, s} \hookrightarrow (\partial M)_{\V, t}$ and $M_{\V, s} \hookrightarrow M_{\V, t}$, respectively. 
    Then
    \begin{align*}
        \dim(\ker(\psi_{st})) &= \dim(\ker(\phi_{st}))\,, \\
        \dim\Big(H_0((\partial M)_{\V, t}) / \image(\psi_{st})\Big) &= \dim\Big(H_0(M_{\V, t})/\image(\phi_{st})\Big)\,,
    \end{align*}
    and $\phi_{st}$ is an isomorphism if and only if $\psi_{st}$ is an isomorphism.
\end{lemma}

\begin{proof}
    For each $t < h_{\V}(c)$, let $r^t \colon M_t \to (\partial M)_t$ be the retraction defined in \Cref{lem:def_retract}. 
    The following diagram commutes:
    \begin{equation*}
\begin{tikzcd}
(\partial M)_{\V,s} \arrow[r, hook]       & (\partial M)_{\V,t}  \\
M_{\V,s} \arrow[r, hook] \arrow[u, "r^s"] & M_{\V,t} \arrow[u, "r^t"']
\end{tikzcd}
\end{equation*}
because $r^t \vert_{M_{\V, s}} = r^s$. 
Therefore, the following diagram also commutes:
\begin{equation*}
    \begin{tikzcd}
    H_0((\partial M)_{\V,s}) \arrow[r, "\psi_{st}"] & H_0((\partial M)_{\V,t}) \\
H_0(M_{\V,s}) \arrow[u, "r^s_*"] \arrow[r, "\phi_{st}"']  & H_0(M_{\V,t}) \arrow[u, "r^t_*"']\,. 
\end{tikzcd}
    \end{equation*}
By \Cref{lem:def_retract}, the maps $r_*^t$, $r_*^s$ are isomorphisms, so the claims follow.
\end{proof}

In the following lemmas, we additionally assume that $M$ is a polygon in $\R^2$. 
We denote by $\{w_j\}_{j=0}^{k}$ the set of vertices on $M$. 
Note that, in general, this set properly contains  $\{x_i\}_{i=0}^{m}$, the set of vertices of $\ch$.

\begin{lemma}\label{lem:CC_contains_vertex}
    Suppose that $M$ is a star-shaped planar polygon. Let $\V \in \bS^1$. 
    For all $t \in \R$ and $x \in (\partial M)_{\V t}$, the component $[x]_t$ contains at least one vertex of $M$, where $[x]_t$ denotes the connected component of $(\partial M)_{\V, t}$ that contains $x$.
\end{lemma}

\begin{proof}
    Let $(w_i, w_j)$ be the line segment of $\partial M$ that $x$ lies on. 
    Without loss of generality, assume $h_{\V}(w_i) \leq h_{\V}(w_j)$. Therefore, $h_{\V}(w_i) \leq h_{\V}(x) \leq h_{\V}(w_j)$, so the component $[x]_t$ contains the vertex $w_i$.
\end{proof}

The next lemma says that new homology classes do not appear or disappear unless a new vertex $w$ of $M$ has been added. 

\begin{lemma}\label{lem:cp_is_vertex}
    Suppose that $M$ is a star-shaped planar polygon. Let $\V \in \bS^1$.
    If $t < h_{\V}(c)$ and $h_{\V}(w) \not \in [s, t]$ for all vertices $w$ of $M$, then
    \begin{equation*}
        \phi_{st}\colon H_0(M_{\V, s}) \to H_0(M_{\V, t})
    \end{equation*}
    is an isomorphism, where $\phi_{st}$ is the map induced by the inclusion $M_s \hookrightarrow M_t$.
\end{lemma}

\begin{proof}
    By \Cref{lem:def_retract}, it suffices to show that $\psi_{st}\colon H_0((\partial M)_{\V, s}) \to H_0((\partial M)_{\V, t})$ is an isomorphism, where $\psi_{st}$ is the map induced by the inclusion $(\partial M)_{\V, s} \xhookrightarrow{} (\partial M)_{\V, t}$. For all $p\in \R$ and all points $x \in (\partial M)_{\V, p}$, let $[x]_p$ denotes the connected component of $(\partial M)_{\V, p}$ that contains $x$.
    By \Cref{lem:CC_contains_vertex}, every connected component contains at least one vertex of $M$, so $H_0((\partial M)_{\V, t})$ and $H_0((\partial M)_{\V, s})$ are spanned by $\{[w_{j_i}]_t\}_i$ and $\{[w_{j_i}]_s\}_i$, respectively, where 
    \begin{equation*}
        \{w_{j_i}\}_i = \{\text{vertices } w \mid h_{\V}(w) \leq t\} = \{\text{vertices } w \mid h_{\V}(w) \leq s\}\,.
    \end{equation*}
    Because $(\partial M)_{\V, t}$ and $(\partial M)_{\V, s}$ contain the same vertices, the map $\psi_{st}([w_{j_i}]_s) = [w_{j_i}]_t$ is surjective. 
    If $\psi([w_j]_s - [w_\ell]_s) = [w_j]_t - [w_\ell]_t = 0$, then either $w_j, w_{(j+1) \mod (k+1)}, \ldots, w_\ell \in (\partial M)_{\V, t}$ or $w_\ell, w_{(\ell + 1) \mod (k+1)}, \ldots, w_j \in (\partial M)_{\V, t}$, where $k + 1$ is the number of vertices of $M$. 
    Without loss of generality, assume the latter. Then $w_j, w_{(j+1) \mod (k+1)}, \ldots, w_\ell \in M_{\V s}$ as well, so $[w_j]_s = [w_\ell]_s$. 
    Therefore, $\psi_{st}$ is injective.
\end{proof}

The following lemma says that adding a single vertex can create at most one new homology class and destroy at most one existing homology class.

\begin{lemma}\label{lem:distinct_cps}
    Suppose that $M$ is a star-shaped planar polygon and let $\V \in \bS^1$.
    Let $[s, t]$ be an interval and let $\phi_{st} \colon H_0(M_{\V, s}) \to H_0(M_{\V, t})$ be induced by the inclusion $M_s \xhookrightarrow{} M_t$. 
    If $t < h_{\V}(c)$ and there is exactly one vertex $w$ such that $h_{\V}(w) \in [s, t]$, then $\dim(\ker(\phi_{st})) \leq 1$ and $\dim(H_0(M_{\V, t})/\text{Im}(\phi_{st})) \leq 1$.
\end{lemma}

\begin{proof}
    Let $w_i$ be the vertex of $M$ such that $h_{\V}(w_i) \in [s, t]$. 
    By \Cref{lem:def_retract}, it suffices to show that $\dim(H_0((\partial M)_{\V, t})/\text{Im}(\psi_{st})) \leq 1$ and $\dim(\ker(\psi_{st})) \leq 1$, where $\psi_{st}$ is the map induced by the inclusion $(\partial M)_{\V, s} \xhookrightarrow{} (\partial M)_{\V, t}$. 
    For all $p\in \R$ and all points $x \in (\partial M)_{\V, p}$, let $[x]_p$ denotes the connected component of $(\partial M)_{\V, p}$ that contains $x$.
    
    Let $x$ be any point in $(\partial M)_{\V, t}$. 
    By \Cref{lem:CC_contains_vertex}, every connected component contains at least one vertex of $M$, so $[x]_t = [w_j]_t$ for some vertex $w_j$. 
    If $j \neq i$, then $w_j \in (\partial M)_{\V s}$ because $w_i$ is the only new vertex, so $[x]_t = \psi_{st}([w_j]_s) \in \text{Im}(\psi_{st})$. 
    Otherwise, if $i = j$, then $[x]_t = [w_i]_t$. 
    Therefore, $\dim(H_0((\partial M)_{\V, t})/\text{Im}(\psi_{st})) \leq 1$ because $[w_i]_t$ is the only possible new component.

    A similar argument as in \Cref{lem:cp_is_vertex} shows that if 
    \[w_{(i-1) \mod (k+1)}, w_{(i+1) \mod (k+1)} \in M_{\V s} \, ,\] then $\ker(\psi_{st})$ is spanned by $[w_{(i-1) \mod (k+1)}]_s - [w_{(i+1) \mod (k+1)}]_s$. 
    Otherwise, $\ker(\psi_{st}) = 0$. In either case, $\dim(\ker(\psi_{st})) \leq 1$.
\end{proof}

We now show that there are objects with simple $\Dgm_0$, which thus have trivial geometric monodromy.

\begin{definition}(\cite{erdos1986some})
    A set $P$ of points in $\R^d$ is said to be in \define{general position} if there are no two parallel and distinct lines each passing through pairs of points in $P$. 
\end{definition}

\begin{theorem}\label{thm:no_monodromy}
Let $M \subset \R^2$ be a star-shaped planar polygon whose vertices are in general position. 
Then $\PHT_0(M)$ has trivial geometric monodromy.
\end{theorem}

\begin{proof}
By \Cref{prop:no_monodromy_suff_conditions}, it suffices to show that $M$ has simple $\Dgm_0$. 
For each $\V \in \bS^1$, the persistence diagram $\Dgm_0(M, \V)$ has only one essential class because $M$ is connected, so it suffices to consider the finite classes. 
Let $(b_1, d_1), (b_2, d_2) \in \widetilde{\Dgm}_0(M, \V)$ be two distinct points. We must have $b_i, d_i < h_{\V}(c)$ for each $i$ because $M_{\V, t}$ is connected for all $t \geq h_{\V}(c)$. 
By \Cref{lem:cp_is_vertex}, there are vertices $w_{b_1}$, $w_{b_2}$, $w_{d_1}$, and $w_{d_2}$ of $M$ such that $b_i = h_{\V}(w_{b_i})$ and $d_i = h_{\V}(w_{d_i})$. 
By \Cref{lem:distinct_cps}, $w_{b_1} \neq w_{b_2}$ and $w_{d_1} \neq w_{d_2}$. 
To obtain a contradiction, suppose that $b_1 = b_2$ and $d_1 = d_2$. 
Then $\ip{\V}{w_{b_1} - w_{b_2}} = 0$ and $\ip{\V}{w_{d_1} - w_{d_2}}= 0$, so $\V$ is orthogonal to both (nonzero) vectors $w_{b_1} - w_{b_2}$ and $w_{d_1} - w_{d_2}$. 
Therefore, the line segments $(w_{b_1}, w_{b_2})$ and $(w_{d_1}, w_{d_2})$ are parallel, which contradicts $M$ being in general position.
\end{proof}

We remark that neither hypothesis in \cref{thm:no_monodromy} is necessary, as \cref{ex:genpos_notnec,ex:polygon_notnec} show.

\begin{example}\label{ex:genpos_notnec}
    Let $M$ be a (filled) regular planar polygon. 
    Then, for every $\V\in\bS^1$, $\Dgm_{0}(M,\V)$ has only one point, namely the essential class, and thus has trivial monodromy. 
    Thus, $M$ has simple $\Dgm_0$, even though its vertices are not in general position.
\end{example}

\begin{example}\label{ex:polygon_notnec}
Consider the shapes depicted in \cref{fig:no_wedge}.
Neither of them is a polygon since they each have a segment ``poking out'', but \cref{fig:no_wedge_good} has no points with multiplicity greater than $1$ in its degree-$0$ persistence diagram, while \cref{fig:no_wedge_bad} has infinite directions $\V \in \bS^1$ such that $\Dgm_0(M, \V)$ has a point with multiplicity $2$ (one such direction is depicted in green in \cref{fig:no_wedge_bad}).
\end{example}

\begin{figure}[h]
\centering
\begin{subfigure}[t]{0.5\textwidth}
    \centering\includegraphics[width = .4\textwidth]{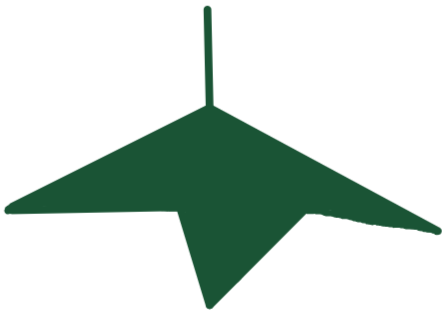}    
    \caption{}
    \label{fig:no_wedge_good}
\end{subfigure}%
\begin{subfigure}[t]{0.5\textwidth}
    \centering
    \includegraphics[width = .3\textwidth]{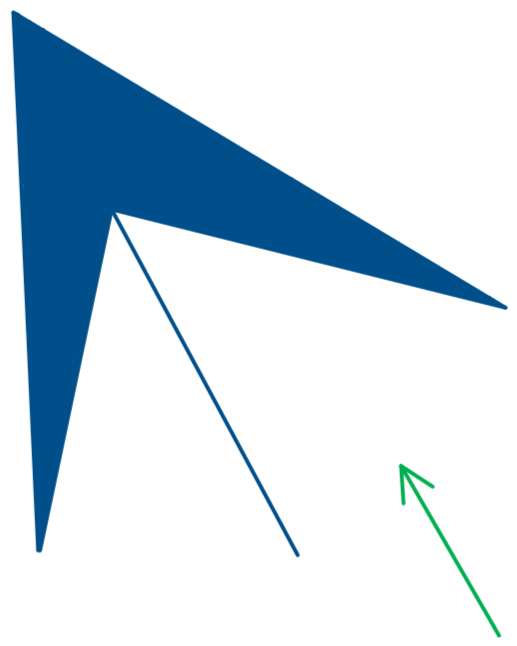}
    \caption{}
    \label{fig:no_wedge_bad}
\end{subfigure} 
\caption{Two non-polygonal star-shaped objects with trivial monodromy, (a) with simple $\Dgm_0$ and (b) with non-simple $\Dgm_0$.}
\label{fig:no_wedge}
\end{figure}

The hypotheses of \cref{thm:no_monodromy} are quite general and are satisfied by many practical examples, such as the leaves in \cref{fig:leaf_examples}.

\section{Conclusions}

\subsection{Summary}
In this paper, we investigate the decompositions of degree-$0$ persistent homology transform for different shapes, focusing on star-shaped objects in $\mathbb{R}^2$. 

Star-shaped objects can be decomposed into smaller, simpler regions (\textit{sectors}) defined by the convex hull and a center of the object. 
We demonstrate that the PHT of star-shaped objects can be decomposed into that of its sectors (\Cref{thm:sector_decomposition}). 

Moreover, we show that star-shaped objects in $\R^2$ with \textit{simple} persistence diagrams (\cref{def:simple_PD}) exhibit trivial geometric monodromy (\cref{prop:no_monodromy_suff_conditions}). 
This is not true in general; see \Cref{ex_monodromy} for an example of a spiral shape whose PHT has non-trivial monodromy. 
The presence of non-trivial monodromy in the context of persistence diagrams is important as it prevents the separation of persistent homology classes into distinct local classes, each living in a subset of $\bS^1$.

In addition, we provide a counterexample of a star-shaped object in $\R^3$ whose short-lived classes are born in one sector and die in another (\cref{no_higherdim_LV}). 
This example shows that the PHT of certain star-shaped objects in $\R^3$ cannot be decomposed into sectors as persistent homology classes survive outside the sectors. 
It is therefore not straightforward to generalize the result of \cref{prop:no_monodromy_suff_conditions} in $\R^3$.

We further investigate the subset of star-shaped objects with simple PDs, showing that star-shaped planar polygons whose vertices are in general position have simple PDs (\cref{thm:no_monodromy}). Real-life data that satisfy these conditions can be found, for example, in leaf shapes (\cref{fig:leaf_examples}).

\subsection{Future work}

There are several directions in which this work can be extended.
First of all, our main results were restricted to shapes in $\R^2$.
In \Cref{sec:alg_dec}, we restricted to $\R^2$ because the Mayer--Vietoris argument required that the sublevelsets of each sector intersected at most two other sectors; this is not true in higher dimensions, so the proof of \Cref{thm:sector_decomposition} does not generalize to $\R^d$. 
We know from \cref{no_higherdim_LV} that the latter result cannot be extended in higher dimensions, but there may be families of shapes for which we can prove, under further hypotheses, a generalization of \cref{thm:sector_decomposition}. 
Moreover, one could focus on different classes of non-star-shaped objects and try to extend first the definition of sectors and then the decomposition into sectors for such objects. 
Depending on the type of objects, this may entail studying higher-degree PHT in addition to degree-$0$ PHT.

Further directions involve the generalizations of the results we presented in \cref{sec:monodromy}. 
Part of the results hold naturally in $\R^d$. 
However, the two main results of this section hold only in $\R^2$, as discussed in \cref{ex:higher_spiral}.
Therefore, being star-shaped is not enough in $\R^d$ to ensure trivial monodromy and additional conditions need to be investigated.

\cref{thm:sector_decomposition} shows that we can decompose the persistence modules in each direction, and, together with \cref{thm:no_monodromy}, this means we can decompose the persistent homology transform into a circle variant of \textit{vineyard modules} \cite{turner2023}. 
These latter were introduced to study the PHT as a whole, i.e. not just pointwisely as a function, as we did here, but rather as a functor, with transition morphisms given by \textit{$\varepsilon$-morphisms}.
Future directions involve understanding the algebraic decomposition (as opposed to geometric decomposition) of the PHT of Euclidean subsets.

\paragraph{Acknowledgements.}
The authors thank the Women in Computational Topology Network (WinCompTop), the \'Ecole Polytechnique F\'ed\'erale de Lausanne (EPFL), 
the Bernoulli Center, and the US National Science Foundation through award number $2317401$ for supporting the third Workshop for Women in Computational Topology in July 2023, where this project began.
B.G. was partially supported by the Austrian Science Fund (FWF) P 33765-N. 
A.H. was supported by NSF grant DMS-2303402.
L.K. was supported by funding to the Blue Brain Project, a research center of the EPFL, from the Swiss government's ETH Board of the Swiss Federal Institutes of Technology.
S.M. was funded in part by NSF CCF-2106578 and NIH NIGMS-R01GM135929.
K.T. was supported by an Australian Research Council Discovery Early Career Award (project number DE200100056) funded by the Australian Government.

\bibliographystyle{plainurl}
\bibliography{bibliography}

\end{document}